%% file: wpmain.tex
\documentclass[12pt,reqno]{article}
\usepackage{amsmath,amsfonts,pinlabel,amsthm,mathptmx,amssymb}
\input{amsmacs}

\input{elcmacs}

\newcommand{\cx}{{\mathbb C}}
\renewcommand{\half}{{\mathbb H}}

\newcommand{\reals}{{\mathbb R}}

\renewcommand{\bold}[1]{\medskip \noindent {\bf \boldmath #1
                        }\nopagebreak[4]}

\newcommand{\bdry}{\partial}

\newcommand{\closure}{\overline}
\newcommand{\compos}{\circ}

\newcommand{\includesin}{\hookrightarrow}
\usepackage{amssymb}

\renewcommand{\Mod}{\mbox{\rm Mod}}

\renewcommand{\PSL}{\mbox{\rm PSL}}

\newcommand{\CAT}{\mbox{\rm CAT}}

\newtheorem*{namedtheorem}{\theoremname}
\newcommand{\theoremname}{testing}
\newenvironment{named}[1]{\renewcommand{\theoremname}{#1}\begin{namedtheorem}}{\end{namedtheorem}}

\newcommand{\calA}{{\mathcal A}}

\newcommand{\calC}{{\mathcal C}}

\newcommand{\calE}{{\mathcal E}}
\newcommand{\calF}{{\mathcal F}}

\newcommand{\calK}{{\mathcal K}}
\newcommand{\calL}{{\mathcal L}}
\newcommand{\calM}{{\mathcal M}}

\newcommand{\calO}{{\mathcal O}}
\newcommand{\calP}{{\mathcal P}}

\newcommand{\calS}{{\mathcal S}}
\newcommand{\calT}{{\mathcal T}}

\newcommand{\ml}{{\calM \calL}}
\newcommand{\pl}{{\calP \calM \calL}}
\newcommand{\el}{{\calE \calL}}

\begin{document}
\title{\bf Asymptotics of Weil-Petersson geodesics II: 
bounded geometry and unbounded entropy}
\author{\sc Jeffrey Brock, Howard Masur, and Yair Minsky\thanks{Each
    author partially 
supported by the NSF.  J. Brock was partially supported by a John Simon
Guggenheim Foundation Fellowship.}}
% \address{Brown University}
% \author{Howard A. Masur}
% \address{University of Illinois at Chicago}
% \author{Yair N. Minsky}
% \address{Yale University}
% \date{\today}
% \thanks{Partially supported by NSF grant DMS-9971596, ...}
% \maketitle

\maketitle

\begin{abstract}
  We use ending laminations for Weil-Petersson geodesics to establish
  that bounded geometry is equivalent to bounded combinatorics for
  Weil-Petersson geodesic segments, rays, and lines.  Further, a more
  general notion of {\em non-annular bounded combinatorics}, which allows
  arbitrarily large Dehn-twisting, corresponds to an equivalent
  condition for Weil-Petersson geodesics.  As an application, we show
  the Weil-Petersson geodesic flow has compact invariant subsets with
  arbitrarily large topological entropy.
\end{abstract}

\renewcommand\marginpar[1]{} % Kill marginpars for the final version

\setcounter{tocdepth}{1}
\tableofcontents

%%%%%%%%%%%%%%%%%%%%%%%%%
% PAPER SPECIFIC MACROS %
%%%%%%%%%%%%%%%%%%%%%%%%%

% WP distance
\newcommand{\dwp}{d_{{\mathrm W\kern-1pt P}}}
\newcommand{\WP}{{\mathrm W\kern-1pt P}}

% WP distance on the completion
\newcommand{\dwpbar}{d_{\overline{\mathrm W\kern-1pt P}}}

\newcommand{\Tsum}{\sum\nolimits'}

% WP ray
\newcommand{\ray}{\mathbf{r}}

% WP geodesic
\renewcommand{\geod}{\mathbf{g}}

% WP Teichmuller geodesic
\newcommand{\tgeod}{\mathbf{h}}

% Ending lamination
\newcommand{\elam}{\lambda}

\renewcommand{\Teich}{\mbox{\rm Teich}}

% WP completion
\newcommand{\compl}{\closure{\Teich(S)}}

\newcommand{\BMMI}{Brock:Masur:Minsky:asymptoticsI}

\newcommand{\sys}{\operatorname{sys}}
\newcommand{\G}{\mathbf{G}}
%%%%% THE BODY OF THE PAPER STARTS HERE

\input{intro}

\input{prelim}

\input{bgimpliesbc}

\input{bcimpliesbg}

\input{entropy}

\bibliographystyle{math}
\bibliography{math}

{\sc \small
 \bigskip

\noindent Brown University \bigskip

\noindent University of Chicago  \bigskip

\noindent Yale University

}

\end{document}

%% file: amsmacs.tex
\catcode`\@=11

\long\def\@savemarbox#1#2{\global\setbox#1\vtop{\hsize\marginparwidth 
  \@parboxrestore\tiny\raggedright #2}}
\newcommand\lref[1]{\ref{#1}%
\@ifundefined{r@DisplaY #1}{}{ (#1)}}% Prints label as well as
\newcommand\fakelabel[2]{\@bsphack\if@filesw {\let\thepage\relax
   \newcommand\protect{\noexpand\noexpand\noexpand}%
\xdef\@gtempa{\write\@auxout{\string
      \newlabel{#1}{{#2}{\thepage}}}}}\@gtempa
   \if@nobreak \ifvmode\nobreak\fi\fi\fi\@esphack}

\catcode`\@=12

\def\Empty{}
\newcommand\oplabel[1]{
  \def\OpArg{#1} \ifx \OpArg\Empty {} \else
        \label{#1}
  \fi}
\newtheorem{theoremSt}{Theorem}[section]

\newtheorem{exampleSt}[theoremSt]{Example}
\newtheorem{exerciseSt}[theoremSt]{Exercise}

\newcommand\MakeStEnv[1]{
  \newenvironment{#1}[1]{%    environment without explicit label
  \begin{#1St} \oplabel{##1}%
  \global\def\CrntSt{\thetheoremSt}%
}{ 
  \end{#1St} }
  \newenvironment{#1+}[1]{%   environment with explicit label
  \begin{#1St} \label{##1}%
  \label{DisplaY ##1}%
  \global\def\CrntSt{\thetheoremSt}%
  \def\Labl{##1}\ifx\Labl\Empty{} \else {\em (\Labl)\,}\fi%
}{ 
  \end{#1St} }
}
\MakeStEnv{theorem}
\MakeStEnv{corollary}
\MakeStEnv{proposition}
\MakeStEnv{lemma}
\MakeStEnv{definition}
\MakeStEnv{conjecture}
\MakeStEnv{problem}
\MakeStEnv{question}

\long\def\state#1#2{
\medskip\par\noindent
{\bf #1} 
{\it #2}
\par\medskip
}

\long\def\realfig#1#2#3{
\begin{figure}[htbp]
{\sf
\centerline{\psfig{figure=#1,height=#3}}
\caption[#1]{#2}
\oplabel{#1}}
\end{figure}}

\newlength{\saveu}

\newenvironment{pf*}[1]{%
 \begin{proof}[#1]%
}{ 
 \end{proof}
}

\newcommand{\finishproof}[1]{ 
  \def\FPArg{#1}
  \ifx\FPArg\Empty
        \newcommand\FPArg{\CrntSt}  \fi
  \smallbreak\noindent\makebox[\textwidth]{\hfill\fbox{\FPArg}}
  \medbreak\noindent
}

\newcommand\AAA{{\mathcal A}}

\newcommand\CC{{\mathcal C}}

\newcommand\FF{{\mathcal F}}

\newcommand\KK{{\mathcal K}}
\newcommand\LL{{\mathcal L}}
\newcommand\MM{{\mathcal M}}
\newcommand\NN{{\mathcal N}}

\newcommand\PP{{\mathcal P}}

\newcommand\SSS{{\mathcal S}}

\newcommand\PMF{{\PP\kern-2pt\MM\FF}}
\newcommand\ML{{\MM\LL}}
\newcommand\PML{{\PP\kern-2pt\MM\LL}}

\newcommand\half{{\textstyle{\frac12}}}

\newcommand\Mod{\operatorname{Mod}}

\newcommand\ep{\epsilon}

\newcommand\union{\cup}
\newcommand\intersect{\cap}
\newcommand\bbR{{\mathord{\text{I\kern-2pt R}}}}        % Fake blackboard
\newcommand\bbH{{\mathord{\text{I\kern-2pt H}}}}        % Fake blackboard

\newcommand\Z{{\mathbb Z}}
\newcommand\R{{\mathbb R}}

\newcommand\PSL[1]{\text{PSL}_{#1}}

\newcommand\bigrightarrow[1]{\hbox to #1{\rightarrowfill}}
\newcommand\bigleftarrow[1]{\hbox to #1{\leftarrowfill}}

\newcommand\boundary{\partial}
\newcommand\semidir{\mathrel{\hbox{\vrule depth-.03ex height1.1ex\kern-0.15em$\times$}}}

\newcommand\til{\widetilde}

\newcommand\tr{\operatorname{tr}}

\newcommand{\diam}{\operatorname{diam}}

\numberwithin{equation}{section}

%% file: elcmacs.tex
\newcommand{\fsubd}{\mathrel{{\scriptstyle\searrow}\kern-1ex^d\kern0.5ex}}
\newcommand{\bsubd}{\mathrel{{\scriptstyle\swarrow}\kern-1.6ex^d\kern0.8ex}}
\newcommand{\fsubeq}{\mathrel{\raise-.7ex\hbox{$\overset{\searrow}{=}$}}}
\newcommand{\bsubeq}{\mathrel{\raise-.7ex\hbox{$\overset{\swarrow}{=}$}}}

\newcommand{\base}{\operatorname{base}}

\newcommand{\EL}{\mathcal{EL}}

\newcommand{\tsh}[1]{\left\{\kern-.9ex\left\{#1\right\}\kern-.9ex\right\}}

\newcommand\geod{\operatorname{\mathbf g}}

\newcommand\twist{\operatorname{tw}}

\newcommand\Teich{{\mathcal T}}

\newcommand\interior{{\rm int}}

%% file: intro.tex
\section{Introduction}
\label{intro}

This paper is the second in a series analyzing the large-scale
behavior of geodesics in the Weil-Petersson metric on Teichm\"uller
space.  In the first paper,
\cite{\BMMI}, we defined a notion of an {\em
  ending lamination} for a Weil-Petersson geodesic ray, and gave a
parametrization of the geodesic rays
based at a fixed point $X \in \Teich(S)$ that recur to the thick part
of Teichm\"uller space in terms of their ending laminations as points
in the Gromov-boundary of the curve complex.  Our main
goal in the present discussion is to describe a connection between
ending laminations of {\em bounded type}, and the control they give over
the trajectories of the rays to which they are associated.

Some of this is in direct analogy with {\em Teichm\"uller} geodesic
rays. For these rays the notion of an ending lamination is explicit in
the definition, and many of the questions we ask already have
well-understood answers.  The lack of a good description of the
behavior of the hyperbolic structure of surfaces that lie along a
Weil-Petersson geodesic has kept a full understanding of the large
scale behavior of geodesics out of reach.

% \begin{conjecture}{EL vanishes}
%   The ending lamination $\elam(\ray)$ is the union of supports of all
%   measured laminations $\mu$ whose lengths go to 0 along $\ray$.
% \end{conjecture}

% \begin{conjecture}{EL structure}
%   The lamination $\lambda(\ray)$ is filling if and only if $r$ is
%   unbounded and not asymptotic to a stratum.

%   If $S(\elam(\ray))$ is the subsurface filled by $\elam(\ray)$ then
%   $\boundary S(\elam(\ray))$ is a sublamination of $\elam(\ray)$.
% \end{conjecture}
% In particular, with Conjecture \ref{EL vanishes}, we expect $\boundary
% S(\elam(\ray))$ to have vanishing length along $\ray$. In the complement of $S(\elam(\ray))$
% we expect the geometry to stabilize:
% \begin{conjecture}{stable on complement}
% If $X$ is a non-annular component of $S\setminus S(\elam(\ray))$ then
% the projection of $\ray(t)$ to $\Teich(X)$ has a limit as $t\to\infty$.
% \marginpar{what about annular?}
% \end{conjecture}

%% We will define a notion of 
%% {\em bounded combinatorics} for a geodesic lamination, using the
%% ``subsurface  projections'' of Masur-Minsky
%% \cite{Masur:Minsky:CCII}, which is analogous to bounded continued
%% fraction expansions for slopes of torus foliations.
%% We will also define the notion of {\em bounded geometry} for a ray. 

%% The main result of this paper indicates the equivalence of these
%% notions for Weil-Petersson geodesics.

The main result of this paper is the equivalence of {\em
  bounded geometry} for a Weil-Petersson geodesic, which is just
precompactness of its projection to the moduli space, and {\em bounded
  combinatorics} of its ending laminations, a notion analogous to
bounded continued fraction expansions for slopes of torus foliations.

Specifically, given a bi-infinite Weil-Petersson geodesic, we consider
ending laminations $\elam^+$ and $\elam^-$ associated to its forward and backward trajectories.  To each essential subsurface $Y \subsetneq
S$ that is not a three-holed sphere, there is an associated
coefficient
$$d_Y(\elam^-,\elam^+)$$ that gives a notion of distance between the
projections of the ending laminations $\elam^+$ and $\elam^-$ in the
curve complex $\CC(Y)$.  We say the pair $(\elam^+,\elam^-)$ has
$K$-{\em bounded combinatorics} if there is an upper bound $K >0$ to all such
coefficients.
 \begin{theorem}{theorem:BCGT}{\sc (Bounded Combinatorics Geometrically Thick)}
   For each 
   $K>0$ there is an $\epsilon> 0$ so that if the ending laminations
   of a bi-infinite Weil-Petersson geodesic $\geod$ 
   have $K$-bounded combinatorics then $\geod(t)$ lies in the
   $\epsilon$-thick part for each $t$.
 \end{theorem}
We make precise the notion of bounded combinatorics in
section~\ref{preliminaries} and remark that similar statements hold for geodesic
segments and geodesic rays (see Theorem~\ref{bc implies bg}).

Conversely, constraining the geometry of surfaces along a
Weil-Petersson geodesic forces a bound to the combinatorics of the
ending laminations.
\begin{theorem}{theorem:TGCB}{\sc (Thick Geodesics Combinatorially Bounded)} 
 Given $\epsilon>0$ there
 is a $K>0$ so that if $\geod$ is any bi-infinite geodesic in the
 $\epsilon$ thick part of $\Teich(S)$, then the combinatorics of the
 ending laminations associated to its ends are $K$-bounded.
\end{theorem}

%The theorem is a consequence of Theorem~\ref{bg near teich} which
%guarantees that a cobounded Weil-Petersson geodesic fellow-travels a
%Teichm\"uller geodesic in the Teichm\"uller metric.  We deduce the
%above combinatorial consequence in 
%Corollary~\ref{bg implies bc}, after an application of a theorem of
%Rafi (see \cite{Rafi:Teich:short}) for Teichm\"uller geodesics.  

As part of the analysis we also have the
following fellow travelling result for Teichm\"uller geodesics.
\begin{theorem}{bg near teich intro} For all $\epsilon>0$ there is a
  $D>0$ so that each bi-infinite $\epsilon$-thick Weil-Petersson
  geodesic $\geod$ lies at Hausdorff-distance $D$, in the
  Teichm\"uller metric,  from a unique
  Teichm\"uller geodesic $\tgeod$. 
\end{theorem}
See Theorem~\ref{bg near teich} for a more precise formulation.

\bold{Remark:}
After the announcement of the main results of this paper in the fall
of 2007, Hamenst\"adt gave an elegant alternative proof of
Theorems~\ref{theorem:BCGT} and~\ref{theorem:TGCB}
\cite{Hamenstaedt:compact:wp} via Teichm\"uller geodesics.

% Conjecturally:

% \begin{conjecture}{bounded structure}
%   A WP-geodesic ray $r$ has bounded geometry if and only if $\ep(r)$
%   has bounded combinatorics.
% \end{conjecture}

% \sotto{This is true if you import the recurrence discussion via a
%   limiting argument} 

% The notion of bounded combinatorics can also be formulated for
% Teichm\"uller geodesics, and was related in that setting to bounded
% geometry by Rafi \cite{rafi:thesis,rafi:short}.  
% Putting these notions
% together we can also formulate this conjecture about the relation
% between Teichm\"uller geodesics and WP-geodesics:

% \marginpar{this is false, I think -- counterexample?}

% \begin{conjecture}{fellow travelers}
% Any Weil-Petersson geodesic $g$ is a fellow traveler, in the WP
% metric, with some Teichm\"uller geodesic $g'$. If $g$ has bounded
% geometry then $g$ and $g'$ are fellow travelers in the Teichm\"uller
% metric as well. 
% \end{conjecture}
% We will obtain in this paper some partial results on these conjectures in
% various special cases. 

\bold{The case of non-annular bounds.}  
%% \marginpar{I think this paragraph maybe should be removed: H}
%% \marginpar{Not sure: Y}
%% An important qualitative
%% difference between the coarse geometry of the mapping class group, 
%% the Teichm\"uller metric, and
%% the Weil-Petersson metric, involves the behavior of
%% Dehn-twists.  While Dehn-twists about simple closed curves have unbounded
%% growth in the mapping class group \cite{Farb:Lubotzky:Minsky}
%% and in the Teichm\"uller metric (linear and logarithmic, respectively)
%% they have {\em bounded diameter orbits} in the Weil-Petersson metric
%%  (see \cite{Brock:wp}).
%% \marginpar{good or bad idea to mention Teich?}
 As an intermediate  step in the proof of Theorem~\ref{theorem:BCGT} we start with the weaker  assumption  of {\em non-annular bounded
   combinatorics}, a criterion considered,
 for example,  by Mahan Mj
(see \cite{Mitra:ibounded}) in the context of Kleinian groups, where the
 coefficients $d_Y(\elam^+,\elam^-)$ are bounded only for essential
 subsurfaces $Y \subsetneq S$ that are 
%%neither three-holed spheres
%% (where such distance is not defined) or 
 not annuli.  This assumption allows for the possibility of
 arbitrarily large {\em relative twisting} of the ending laminations
 $\elam^+$ and $\elam^-$ around various closed curves.

 With this weaker assumption, we obtain a stability theorem for
 quasi-geodesics in the {\em pants graph}, a combinatorial model for
 the Weil-Petersson metric (\cite{Brock:wp}).  The pants graph $P(S)$,
 introduced by Hatcher and Thurston, encodes isotopy classes of pants
 decompositions of the surface $S$ as its vertices, with edges joining
 vertices whose corresponding pants decompositions are related by
 certain elementary moves.  By Theorem~1.1 of \cite{Brock:wp}, there
 is a quasi-isometry
$$Q\colon \Teich(S) \to P(S)$$that associates to each $X \in
\Teich(S)$ a shortest {\em Bers pants decomposition} for $X$.

\state{Theorem \ref{i bounded}.}{
%%%\begin{theorem}{theorem:i bounded intro}{
{\sc (Stability without Annuli)}
Let $F \colon [0,T] \to P(S)$ be a $K$-quasi-geodesic, and let $F(0) =
Q_-$ and $F(T) = Q_+ $ denote its endpoints.  If $Q_-$ and $Q_+$
satisfy the non-annular bounded combinatorics condition, then $F(t)$
fellow travels a hierarchy path in $P(S)$.
}

Hierarchy paths will be discussed  in Section~\ref{preliminaries}, and
the precise statement of Theorem \ref{i bounded} appears in \S4. 
Stability of quasi-geodesics, standard in the setting of
$\delta$-hyperbolic metric spaces, only holds generally in $P(S)$ for
low-complexity cases when the dimension $\dim_\cx(\Teich(S))$ is $1$
or $2$, (see \cite{Brock:Farb:rank}).  By the main result of
\cite{Brock:Masur:pants}, a relative version holds for
$\dim_\cx(\Teich(S)) =3$.

Stability is also natural question in the context of Weil-Petersson
geometry, as a Weil-Petersson geodesic determines a quasi-geodesic in
the pants graph via the quasi-isometry.  The corresponding condition
for Weil-Petersson geodesics is more difficult to formulate, but it
allows the possibility for geodesics to approach boundary strata over
very small intervals of time by twisting.  We will not need a result
of this type, but the phenomenon of large twisting along short
intervals approaching boundary strata, suggests why a bound on the
amount of twisting, a condition not part of non-annular bounded
combinatorics, but part of the assumption of $K$-bounded
combinatorics, will be necessary to prove Theorem~\ref{theorem:BCGT}.

\smallskip
\noindent {\bf Topological entropy.}  In \cite{\BMMI}
we showed how ending laminations can be employed to understand
fundamental features of the topological dynamics of the Weil-Petersson
geodesic flow on the unit tangent bundle to moduli space (see
\cite[Thms. 1.8 and 1.9]{\BMMI}).  In this
paper, we show how the finer combinatorial features of the ending
laminations described above provide for further understanding of the
flow.  In particular we show
\newcommand{\entropy}{
{\sc (Topological Entropy)}
  There are compact flow-invariant subsets of $\calM^1(S)$ of
  arbitrarily large topological entropy.
}
\begin{theorem}{theorem:entropy}
\entropy
%%statement above
\end{theorem}
Here, $\calM^1(S) = T^1 \Teich(S) / \Mod(S)$ represents the quotient
of the unit tangent bundle to $\Teich(S)$ by the action of the mapping
class group $\Mod(S)$.  We note that due to the fact that the
Weil-Petersson geodesic flow is well defined for all time only on the
lifts of bi-infinite geodesics in $\calM(S)$ to $\calM^1(S)$, topological entropy is
not well defined on the whole of $\MM^1(S)$.  Nevertheless, the
topological entropy for compact invariant subsets of the
Weil-Petersson geodesic flow sits in strong contrast to the
topological entropy for compact invariant subsets of the Teichm\"uller
geodesic flow, which approach a positive supremum equal to the real
dimension of the Teichm\"uller space in question (see
\cite{Hamenstaedt:compact}).  Theorem \ref{theorem:entropy} follows
from the following unboundedness result for the growth rate
$$p_\varphi(\calK) = \liminf_{L \to \infty} \frac{\log n_\calK(L) }{L}$$
of the number $n_\calK(L)$ of closed orbits for the geodesic flow  of
length at most $L$ in the invariant subset $\calK$.

\newcommand{\growthrate}{ Given
  any $N>0$, there is a compact Weil-Petersson geodesic flow-invariant
  subset $\calK \subset \calM^1(S)$ for which the asymptotic growth
  rate $p_\varphi(\calK)$ for the number of closed orbits in $\calK$
  satisfies
$$
p_\varphi(\calK) 
 \ge N.
$$
}
\begin{theorem}{growth rate}{\sc (Counting Orbits)}
\growthrate
\end{theorem}

\bold{Plan of the paper.}  The paper makes considerable use of the
technology of {\em hierarchies} of geodesics in the curve complex
$\calC(S)$ and in the curve complexes $\calC(W)$ of subsurfaces $W
\subset S$ developed in \cite{Masur:Minsky:CCII}.  Here, we axiomatize
the idea of a {\em resolution} of such a hierarchy into a notion of
{\em hierarchy path} in Section~\ref{preliminaries}, where we also
introduce other terminology we will need.  Section~\ref{BG implies BC}
%%%uses the notion of quasi-convexity of Lipschitz paths in $\Teich(S)$
establishes 
Theorem~\ref{theorem:TGCB} by a compactness argument.  The same compactness argument also shows   that a Weil-Petersson geodesic in the thick part of
$\Teich(S)$ fellow travels a Teichm\"uller geodesic in the
Teichm\"uller metric.    Section~\ref{BC implies BG} establishes Theorem
~\ref{i bounded},
and uses this to show Theorem~\ref{theorem:BCGT}, after an analysis of
the combinatorial behavior of bounded length geodesics using recent
work of Wolpert (Theorem~\ref{geodesic limits}).  Finally, in
Section~\ref{entropy}, we apply the results of Section~\ref{BC implies
  BG} to establish that there are compact geodesic-flow-invariant
subsets of arbitrarily large topological entropy.

\bold{Acknowledgements.}  The authors thank the Mathematical Sciences
Research Institute for their hospitality while much of the work in
this paper was completed.  The first author thanks Yale University for
its hospitality and the John Simon Guggenheim Foundation for its
generous support.  We thank Ursula Hamenst\"adt and Scott Wolpert for
illuminating conversations.

%% file: prelim.tex
\section{Preliminaries}
\label{preliminaries}
In this section we review terminology and background, setting
notation we will use.

\bold{Teichm\"uller space and its metrics.}  If $S$ is a compact
surface of negative Euler characteristic, the {\em Teichm\"uller
  space} $\Teich(S)$ denotes the space of finite-area hyperbolic
structures on $\interior(S)$ up to isotopy. By default we consider the
Weil-Petersson metric on $\Teich(S)$, which is defined via an
$L^2$-norm on cotangent spaces given by 
$$\| \varphi \|^2_{\WP} = \int_X \frac{|\varphi|^2}{\rho_X}$$ where $\varphi
\in T_X^*\Teich(S)$ is a holomorphic quadratic differential on $X$ and
$\rho_X$ denotes the Poincar\'e metric on $X$.  The induced Riemannian
metric $g_{\WP}$ has been much studied by many authors, and we
focus here on properties of its synthetic and geometry and distance
function $d_{\WP}(.,.)$.

We will occasionally refer to the Teichm\"uller metric, a Finsler
metric arising from an $L^1$ norm on cotangent spaces.  Its distance
function $d_T$ measures the infimum over all quasiconformal maps in
the natural isotopy class of the quasi-conformal dilatation.

\bold{Pants decompositions and markings.}  A {\em pants decomposition}
$P$ in a surface $S$ of finite type is a maximal collection of isotopy
classes of disjoint, homotopically distinct, homotopically nontrivial,
non-peripheral simple closed curves.  A {\em marking} $\mu$ consists
of a pants decomposition $P$ which is the {\em base} of the marking,
written $\text{base}(\mu)$, together with a collection of isotopy
classes of {\em transversals} one for each curve in
$\text{base}(\mu)$.  For each $\alpha\in P$ the transversal curve
$\alpha'$ intersects no curve in $P\setminus\alpha$, and intersects
$\alpha$ a minimal number of times subject to this condition
(i.e. once or twice depending on the topological type of
$S\setminus(P\setminus\alpha)$).  We will frequently blur the
distinction between curves and their isotopy classes, as there is a
unique geodesic representative in each isotopy class.

\bold{The Bers constant.}  Given $S$ of negative Euler characteristic,
we denote by $L_S >0$ the constant so that for each $X \in \Teich(S)$,
there is a {\em Bers pants decomposition} $P_X$ of $X$ determined by
closed geodesics on $X$ whose lengths are bounded by $L_S$.  The
isotopy classes of closed geodesics determining $P_X$ are called {\em
  Bers curves} for $X$ (see \cite{Buser:book:spectra}).  For each
pants curve choose a minimal length transversal.  The resulting
marking is called a {\em Bers marking} and is denoted by $\nu_X$.  By
the collar lemma \cite{Buser:book:spectra}, there are a bounded number
of Bers pants decompositions and Bers markings on a given $X$.
We call a curve that arises in a Bers pants decomposition $P_X$ for
$X$ a {\em Bers curve} for $X$.

\bold{The complex of curves.}  The complex of curves $\calC(S)$ serves
to organize the isotopy classes of essential non-peripheral simple
closed curves on $S$.  Each is associated to a vertex of $\calC(S)$,
and $k$-simplices are associated to families of $k+1$ distinct isotopy
classes that can be realized pairwise disjointly on $S$ (there is an
exception for 1-holed tori and 4-holed spheres, where 1-simplices
correspond to pairs of vertices realized with intersection number 1
and 2, respectively). We make $\calC(S)$ into a metric space by 
making each simplex Euclidean with sidelength 1, and letting $d_\calC$
be the induced path metric. 
By the main result of
\cite{Masur:Minsky:CCI}, $(\calC(S),d_\calC)$ is
$\delta$-hyperbolic, in the 
sense of Gromov.  

A $\delta$-hyperbolic space 
 carries a natural {\em Gromov boundary}. In our setting of a
 path-metric space, points in this boundary are 
asymptote classes of  quasi-geodesic rays, where two rays are
asymptotic if their Hausdorff distance is finite. 
%% , namely, asymptote
%% classes of infinite sequences, where two sequences $\{x_n\}$ and
%% $\{y_n\}$ are equivalent if their Gromov inner product 
%% $\langle x_n,y_n\rangle_{x_0}\to\infty$.  Here $x_0$ is some fixed point.   
%where two rays are
%{\em asymptotic} if they lie in a uniformly bounded Hausdorff
%distance.  
Klarreich showed \cite{Klarreich:boundary} (see also
\cite{Hamenstaedt:boundary}) that the Gromov boundary of $\CC(S)$ is
identified with the space $\el(S)$ of {\em ending laminations} on $S$.
We define $\el(S)$ by starting with Thurston's measured lamination
space $\ml(S)$, considering the subset of those laminations that {\em
  fill the surface} (namely, laminations $\mu\in \ml(S)$ so that every
essential simple closed curve $\gamma$ satisfies $i(\mu,\gamma) >0$)
and forgetting the transverse measure on these laminations. The
resulting quotient space is $\el(S)$, with the quotient, or
``measure-forgetting,'' topology.  Convergence from within $\CC(S)$ to
$\EL(S)$ is also defined using this topology, considering $\CC(S)$ as
a subset of $\ml(S)$ modulo measures.

The mapping class group, $\Mod(S)$, of orientation preserving
homeomorphisms modulo those isotopic to the identity, acts naturally
on $\calC(S)$ via its action on the essential simple closed curves on
$S$.  Given a simplex $\sigma \in \calC(S)$, we denote by
$$
\twist(\sigma) < \Mod(S)
$$
the free abelian group generated by Dehn twists about the curves
represented by the vertices of $\sigma$.

\bold{The pants graph and marking graph.}  Central to our discussion
here will be the quasi-isometric model for the Weil-Petersson metric
obtained from the graph of pants decompositions on surfaces.  The
isotopy class of a pants decomposition $P$ of $S$ corresponds to a
vertex of $P(S)$, and two vertices corresponding to pants
decompositions $P$ and $P'$ are joined by an edge if they differ by an
{\em elementary move}, namely, if $P'$ can be obtained from $P$ by
replacing one of the isotopy classes of simple closed curves
represented in the pants decomposition $P$ with another that
intersects it minimally.  This defines a distance $d_P(\cdot, \cdot)$
in the pants graph.

Then there is a coarsely defined projection map
$$Q:\Teich(S) \to P(S)$$
that associates to each $X \in \Teich(S)$ a Bers pants decomposition
on $X$.
\begin{theorem}{theorem:pants:quasi}
{\rm (\cite[Thm. 1.1]{Brock:wp})}
The map $Q$ is a quasi-isometry.
\end{theorem}

The {\em marking graph} $\widetilde M(S)$ is the graph whose vertices are
markings (as above) and whose edges correspond to elementary moves
which correspond to twists of transversals around pants curves, and
(roughly) interchange of pants curves and transversals (see
\cite{Masur:Minsky:CCII}). We denote the path metric associated to this graph by
$d_{\widetilde M(S)}$. 
The relevant property for us is that $\widetilde M(S)$ is
connected, and is acted on
isometrically and cocompactly by the mapping class group. Note that
$\widetilde M(S)$ ``fibres'' over $P(S)$ by the map that forgets the
transversals. 
\marginpar{Is that what we need?}

\bold{Weil-Petersson geodesics.}  A {\em Weil-Petersson geodesic} is
denoted $\geod:J \to \Teich(S)$ where $J\subset \R$ is an interval,
and $\geod$ is geodesic parametrized by arclength, with respect to the
Weil-Petersson distance $\dwp$.

Let $\alpha=\inf J$ and $\omega=\sup J$.  If $\omega\in J$
(respectively $\alpha\in J$) we say the forward (respectively
backward) end of $\geod$ is closed. If $\omega\notin J$ we require
that $\geod$ cannot be extended past $\omega$
(i.e. $\geod(t)$ exits every compact set in $\Teich(S)$ as 
$t\to\omega$), and we
say the forward end of $\geod$ is open; and similarly for the backward
end.

We call $\geod$ a {\em segment} if $J=[\alpha,\omega]$, a {\em ray} if
$J=[\alpha,\omega)$ or $J=(\alpha,\omega]$, and a {\em line} if
$J=(\alpha,\omega)$.  If $J$  (hence $\geod$) has infinite length,
we call $\geod$ an unbounded ray or line.

If the forward (resp. backward) end of $\geod$ is closed we denote by
$\nu^+(\geod)$ a choice of Bers marking   $\nu_{\geod(\omega)}$
for the surface $\geod(\omega)$.
(resp. $\nu^-(\geod)=\nu_{\geod(\alpha)}$).  
For open ends we have the notion of {\em ending lamination}, which we
will define presently.

\bold{Geodesic length functions and ending laminations for rays.}  To
each isotopy class of essential non-peripheral simple closed curves
there is an associated {\em
  geodesic length function}
$$
\ell_\alpha:\Teich(S)\to\R_+
$$ 
that assigns to each $X \in \Teich(S)$ the arclength $\ell(\alpha^*)$
of the geodesic representative $\alpha^*$ of $\alpha$ on $X$.  Given a
path $\geod(t)$ in $\Teich(S)$,
$\ell_\alpha$ determines a natural
{\em length function} along a geodesic $\geod$:
$$\ell_{\geod,\alpha}(t) = \ell_\alpha(\geod(t)).$$
When $\geod(t)$ is a Weil-Petersson geodesic, it is due to Wolpert
(see \cite{Wolpert:Nielsen}) that the function
$\ell_{\geod,\alpha}(t)$ is strictly convex.

In \cite{\BMMI} we study the following definitions for a 
Weil-Petersson geodesic ray $\ray$.
\begin{definition}{ending measure}
  An {\em ending measure} for a geodesic ray $\ray(t)$ is
  any limit $[\mu]$ in $\pl(S)$, Thurston's space of projective measured laminations,  of the projective classes $[\gamma_n]$
  of any infinite family of distinct Bers curves for $\ray$.
\end{definition}

We pay special attention to simple closed curves whose length decay to zero.
\begin{definition}{pinching curve}
  A simple closed curve $\gamma$ is a {\em pinching curve} for $\ray$ if
 $\ell_{\ray,\gamma}(t) \to 0$
as $t \to \omega$.
\end{definition}

Taking the union of the support of all ending measures together
with the pinching curves for $\ray$ we obtain the {\em ending lamination
$\elam(\ray)$.} That this is in fact a lamination follows from
\cite[Prop. 2.9]{\BMMI}, which states: 

\begin{theorem}{EL def}
  If $\ray(t)$ is a Weil-Petersson geodesic ray, the pinching curves
  and supports of ending measures for $\ray$ have no transverse
  intersection. Hence their union $\elam(\ray)$ is a geodesic
  lamination. 
\end{theorem}

\bold{Ending data.}
For each open end of a geodesic $\geod$ we thus have an ending
lamination; we denote these by $\lambda^+(\geod)$ for the forward
($\omega$) end and $\lambda^-(\geod)$ for the backward ($\alpha$)
end. 

If (say) the forward end of $\geod$ is closed, so that
$\geod(\omega)\in\Teich(S)$, then we let $\nu^+(\geod)$ denote a
Bers marking for $\geod(\omega)$ (if there are several we pick one arbitrarily). Define $\nu^-(\geod)$ similarly. In general we call
$\nu^\pm(\geod)$ or $\lambda^\pm(\geod)$ the {\em ending data} of
$\geod$, and if we do not wish to be specific about whether they are
markings or laminations we use the notation $\nu^\pm$.

\bold{The completion of the Weil-Petersson metric and its strata.}
The failure of completeness of the Weil-Petersson, due to Wolpert and Chu (see
\cite{Wolpert:noncompleteness} and \cite{Chu:noncompleteness}) arises
from finite length  paths in the Weil-Petersson metric that leave every compact set corresponding to
``pinching deformations'' where the length of a family of  simple
closed curves tends to zero.  If the hyperbolic metrics on the
complement of this family of curves converge to finite area hyperbolic
structures on their complementary subsurfaces one may identify this
{\em nodal surface} with the corresponding point in the Weil-Petersson
completion.  Masur showed that the completion is naturally the
augmented Teichm\"uller space (see
\cite{Bers:nodes,Abikoff:degenerating}) obtained by adjoining {\em
  boundary strata} corresponding to (products of) lower dimensional
Teichm\"uller spaces.  

These strata and their adjunction are best understood by extending
geodesic length functions $\ell_\alpha$ to allow their vanishing.
Precisely, if $P$ is a pants decomposition of $S$ then $P$ determines
a maximal simplex in $\CC(S)$, and a {\em frontier space subordinate
  to $P$}, $\SSS_\sigma$, is determined by the vanishing of the length
functions $\{\ell_\alpha = 0 \} $ for the simple closed curves
$\alpha$ representing the vertices of $\sigma$.  The topology on the
union of $\Teich(S)$ with the frontier spaces subordinate to $P$ is
given by {\em extended Fenchel-Nielsen coordinates} for $\Teich(S)
\cup \SSS_\sigma$ in which the length parameters $\ell_\alpha$, are
extended to take values in $\reals_{\ge 0}$, and twist parameters
$\theta_\alpha$ are undefined on $\SSS_\sigma$ for each $\alpha \in
\sigma^0$.  Then the union is topologized by the requirement that
$\ell_\alpha$ vary continuously.  See \cite{\BMMI}
\cite{Wolpert:compl}.

\bold{Hierarchy paths and the distance formula.}  Though the pants
complex gives a coarse notion of distance in the Weil-Petersson
metric (via Theorem \ref{theorem:pants:quasi}), it is not at all clear what form distance minimizing paths may
take.  Nevertheless, a kind of combinatorial formula to estimate pants
distance arises out of consideration of the curve complex $\calC(S)$
and the curve complexes $\calC(Y)$ of subsurfaces $Y \subset S$
considered simultaneously.

First of all, a notion of {\em projection to a subsurface} is defined:
given a proper, essential subsurface $W \subset S$, there is a
projection
$$\pi_W \colon \calC(S) \to \calP(\calC(W))$$
from the curve complex $\calC(S)$ to the power set of the curve
complex $\calC(W)$ as follows: for any $\gamma \in \calC(S)$ with
$\gamma$ isotopic into the complement of $W$, we set $\pi_W(\gamma) =
\emptyset$.  Now assuming $W$ is not an annulus:
If $\gamma$ is isotopic into $W$ then we set
$\pi_W(\gamma) = \gamma$.  Otherwise, after isotoping $\gamma$ to
minimize the number of components of $\gamma\intersect W$, we take for
each arc $a$ of the intersection the boundary components of a regular
neighborhood of $a\union \boundary W$ which are essential curves in
$W$. The union of these is $\pi_W(\gamma)$. 

%% taking the geodesic
%% representative $\gamma^*$ of $\gamma$ and $W^*$ the subsurface
%% isotopic to $W$ with geodesic boundary, we take $\pi_W(\gamma)$ to be
%% the union of the set of essential closed curves in the boundary of a collar neighborhood of
%% $\bdry W^* \cup (\alpha \cap W^*)$ where $\alpha$ is an arc of
%% intersection of $\gamma^*$ with $W^*$.  
% As shown in
% \cite{Masur:Minsky:CCII}, when $\gamma$ intersects $W$ the set
% $\pi_W(\gamma)$ has diameter at most $2$
% (See \cite{Masur:Minsky:CCII}).

The case that $W$ is an annulus is slightly different: here we let
$\CC(W)$ denote the complex whose vertices are arcs connecting the
boundaries of $W$ up to isotopy rel endpoints, and whose edges are
pairs of arcs with disjoint interiors. If $\gamma$ intersects $W$
essentially we lift it to the annular cover associated to $W$, which
we identify with $W$, and let $\pi_W(\gamma)$ be the union of
components of the lift that connect the boundary components of
$W$. Note that $\CC(W)$ is quasi-isometric to $\Z$ and its distance
function measures a coarse form of twisting around the annulus $W$. This is
sometimes called the twist complex of $W$.

The {\em projection distances}
$$d_W(\sigma,\gamma) = \diam(\pi_W(\sigma),\pi_W(\gamma))$$
give a useful notion of the relative distance between simplices
$\sigma$ and $\gamma$ as seen from the subsurface $W$.  We can define
this just as well when $\sigma$ or $\gamma$ are markings. 
Then one has the following Lipschitz property for the projections
$\pi_W$
(see \cite[Lemma 2.3]{Masur:Minsky:CCII}):
\begin{proposition}{Lipschitz projection}
For any simplex $\sigma \in \calC(S)$ and any subsurface $W
\subset S$, if $\pi_W(\sigma) \not= \emptyset$ then we have 
$$\diam_{\calC(W)}(\pi_W(\sigma)) \le 2.$$

Similarly, if $\mu$ and $\mu'$ are pants decompositions or markings
on $S$ differing by an elementary move, 
then we have
 $$d_W(\mu,\mu') \le 4.$$
\end{proposition}
% \marginpar{Has somebody checked this 4? Tricky cases: $S=S_{0,4}$ and
%   $W$ annulus}

There is a strong relationship between the geometry of geodesics in
such curve complexes, projection distances, and a certain type of
efficient path in $P(S)$ called a {\em hierarchy path}.
These considerations, developed in \cite{Masur:Minsky:CCII} 
can be summarized in the following theorem. 

Given $n \ge 0$ let
$[[n]]_M$ denote the quantity
\begin{displaymath}
[[n]]_M = \left\{
\begin{array}{cc}
n & \ \text{if} \ n \ge M, \ \text{and} \\ 0 & \text{otherwise.}
\end{array}\right.  
\end{displaymath}
Further, given $c_1 >1$ and $c_2 >0$ we denote by $\asymp_{c_1,c_2}$
equality up to multiplicative error $c_1$ and additive error $c_2$.
In other words, we write $x \asymp_{c_1,c_2} y$ whenever
$$ \frac{x}{c_1} - c_2 \le y \le c_1 x + c_2.$$
\begin{theorem}{hierarchy paths}{\sc (Hierarchy Paths)}
Given pants decompositions $P_1$ and $P_2$ in $P(S)$, there is a path
$\rho \colon [0,n] \to P(S)$ joining $\rho(0) = P_1$ to $\rho(n) =
P_2$ with the following properties.
%\marginpar{Shouldn't this be a theorem rather than a definition??? -Y}
\begin{enumerate}
\item There is a collection $\{Y\}$ of essential, non-annular
  subsurfaces of $S$, called {\em component domains} for $\rho$, so
  that for each component domain $Y$ there is a connected interval
  $J_Y \subset [0,n]$ with $\bdry Y \subset \rho(j)$ for each $j \in
  J_Y$.
\item There is an $M_1 >0$ so that for each essential subsurface $Y
  \subset S$ with $d_Y(P_1,P_2) >M_1$, $Y$  is a component domain for $\rho$.
\item For component domain $Y$, there is a geodesic $g_Y \subset \calC(Y)$ so
  that for each $j \in J_Y$, there is an $\alpha \in g_Y$ with
  $\alpha \in \rho(j)$.
\item \label{bded image}
There is an $M_2>0$ so that if $J_Y = [t_1,t_2]$ and $t > t_2$ then
  $d_Y(\rho(t),\rho(t_2)) < M_2$ while for $t< t_1$ we have
  $d_Y(\rho(t),\rho(t_1)) < M_2$.
\item 
\label{dist formula} {\sc (The Distance Formula)}
Given any $M_3 \ge M_1$, there exists $c_1>1$ and $c_2>0$,  so that
$$d_P(P_1,P_2) \asymp_{c_1,c_2} 
\Tsum_Y
 \left[[
  d_Y(P_1,P_2) \right]]_{M_3}.$$
where the notation $\Tsum_Y$ indicates that the sum is
  taken over all {\em non-annular} subsurfaces $Y$ including $S$ itself.
\item
  \label{quasigeodesic} There is a $K_H >1$ so that the path $\rho$
  satisfies
\begin{equation}
\frac{1}{K_H}  \le \frac{d_P(\rho(i),\rho(j))}{|j-i|}  \le 1.
\end{equation}

\end{enumerate}
The whole surface $S$ is always a component domain for $\rho$.  The
geodesic $g_S$ is called the {\em main geodesic} for the hierarchy
path $\rho$.
\end{theorem}
Given $P_1$ and $P_2$ in $P(S)$, we denote by $\rho = \rho(P_1,P_2)$ an
arbitrary choice of hierarchy path joining $P_1$ to $P_2$.

These hierarchy paths in $P(S)$ are {\em resolution sequences}, so
called, of the {\em hierarchies without annuli} defined in
\cite[\S8]{Masur:Minsky:CCII}.  
%% We find it convenient to axiomatize
%% their properties into the above form
%% (cf. \cite[Defn. 9]{Brock:Masur:pants}). 
The main construction of \cite{Masur:Minsky:CCII} actually takes place
in the marking graph $\widetilde{M} (S)$, and 
Theorem~\ref{hierarchy paths} can be restated for hierarchy
paths $\rho \colon [0,n] \to \widetilde{M}(S)$ in the marking graph,
where surfaces $Y$ are allowed to be annuli.  In particular, there is 
a distance formula for marking distance $d_{\widetilde{M}}$
which takes
the same form as (5), but the sum is taken over all essential
subsurfaces (see \cite[Thm. 6.12]{Masur:Minsky:CCII}).

There will be instances in the paper where it is appropriate to
consider hierarchy paths in $\rho$ in $P(S)$ and in
$\widetilde{M}(S)$, and we will make it clear from context which is
being considered.  The main construction of \cite{Masur:Minsky:CCII}
also allows for the case when the main geodesic $g_S$ is infinite or
bi-infinite, in which case the geodesic, and hence the hierarchy path,
describes curve systems that are asymptotic to a lamination or pair of
laminations in $\el(S)$ in the measure forgetting topology.

%% class group $\Mod(S)$.  In particular, a pants decomposition $P$ may
%% be extended to a {\em full marking} $\mu$ by associating the
%% additional data of a {\em transversal} $\alpha_\beta$ to each $\beta
%% \in P$, so that $i(\alpha_\beta, \beta') = 0$ for each $\beta' \in P$
%% with $\beta' \not= \beta$, and $\alpha_\beta$ has minimal intersection
%% with $\beta$ among all curves satisfying these conditions.  Then
%% curves in $P$ are called {\em base curves} for the marking $\mu$, and
%% the other curves are called {\em transversals}.
%% unnnn

%% Then two markings $\mu$ and $\mu'$ differ by an {\em elementary
%%   marking move} if $\mu'$ is obtained by 
%% \begin{enumerate}
%% \item replacing a transversal $\alpha$ for $\beta$ by its image under
%%   a right or left Dehn-twist about $\beta$, or
%% \item replacing a base curve $\beta$ for $\mu$ with its transversal $\alpha$, and
%%   choosing a  transversal $\beta'$ for $\alpha$ whose distance from
%%   $\beta$ in $\calA_\alpha$ is minimal.
%% \end{enumerate}
%% (a variant of this description appears in \cite{Masur:Minsky:CCII}).

\bold{Bounded combinatorics.} Let  $\sigma$ and
$\sigma'$  be simplicies in $\calC(S)$  or let them be  markings.  We say $\sigma$ and $\sigma'$  have {\em
  $K$-bounded combinatorics} if for each essential subsurface $W
\subsetneq S$ which is intersected by both $\sigma$ and $\sigma'$, 
the projection distance satisfies
$$d_W(\sigma,\sigma') < K.$$

We note that for any hierarchy path $\rho \colon [0,n] \to P(S)$ or 
hierarchy path $\rho\colon [0,n] \to \widetilde{M}(S)$ whose
endpoints satisfy 
$$d_W(\rho(0),\rho(n)) <K$$ 
an application of the triangle inequality together with
part~(\ref{bded image}) of Theorem~\ref{hierarchy paths} guarantees
that for each $i$, $j$ in $[0,n]$, we have
\begin{equation}\label{bded hierarchy}
d_W(\rho(i),\rho(j)) < K + 2M_2.
\end{equation}

% \sotto{Projections to subsurfaces, distance formula for pants, and for
%   markings, hierarchy (projection to a hierarchy), main geodesic,
%   bounded combinatorics}

\bold{Recurrence for rays}  Given $X \in \Teich(S)$ we let 
$\sys(X)$ denote the length of the shortest closed geodesic in $X$.
A Weil-Petersson geodesic ray $\ray$ is said to be {\em recurrent} if
there is an $\epsilon >0$ and a family of times $t_n \to \infty$ for
which the hyperbolic surface $\sys(\ray(t_n)) > \epsilon$.  It was
shown in \cite{\BMMI} that the ending lamination uniquely determines
the asymptote class of for a recurrent ray.  Precisely, we have
\begin{theorem}{relc} Let $\ray$ and $\ray'$ be two geodesic rays with
  $\elam(\ray) = \elam(\ray')$.  If $\ray$ is recurrent, then $\ray$
  and $\ray'$ are strongly asymptotic.
\end{theorem}
Here, strongly asymptotic refers to the existence of parametrizations
$\ray(s)$ and $\ray'(s)$ (not necessarily by arclength) for which
$d(\ray(s),\ray'(s)) \to 0$.

Given $\epsilon>0$, a geodesic segment, ray, or line $\geod$ is said to
have $\epsilon$-bounded geometry if the length of the shortest closed
geodesic on $\geod(t)$ is bounded below by $\epsilon$ for each $t$ for
which $\geod(t)$ is defined.  
Then we observe that as a direct consequence of Theorem~\ref{relc} we
have the following.
\begin{corollary}{bded elc}
Let $\ray$ be a geodesic ray with $\epsilon$-bounded geometry.
Then if $\ray'$ satisfies $\elam(\ray) = \elam(\ray')$ then $\ray$ and
$\ray'$ are strongly asymptotic.
\end{corollary}

The following Proposition  combining Lemma~2.10 and Corollary~2.12
of \cite{\BMMI} will be useful for our purposes.
\begin{proposition}{proposition:continuity}

  Let $\ray_n \to \ray_\infty$ be a sequence of segments or rays based
  at a fixed $X \in \Teich(S)$ that converge in the visual sphere, and
  assume $\ray_\infty$ has an ending measure $\mu$.  If $\ray_n$ is a
  segment, let $\mu_n$ be a Bers pants decomposition for its endpoint.
  If $\ray_n$ is a ray, let $\mu_n$ be any ending measure or pinching
  curves for $\ray_n$.  Let $\mu' \in \ml(S)$ a representative of any
  limit $[\mu']$ of projective classes $[\mu_n]$ in $\pl(S)$.  Then we
  have $$i(\mu,\mu') = 0.$$

 In particular, if $\mu$ fills the surface, then $|\mu| = |\mu'|$.

%   Let $\ray_n \to \ray_\infty$ be a sequence of segments or rays based
%   at a fixed $X \in \Teich(S)$ that converge in the visual sphere.  
% If $\ray_n$ is a segment, let
%   $[\mu_n]$ be a bounded length pants decomposition at the endpoint
%   whose lengths on $X$ go to $\infty$.  If $\ray_n$ is a ray, let
%   $\mu_n$ be any sequence of ending measures or weighted pinching
%   curves for $\ray_n$.  Let $\mu \in \ml(S)$ a representative of any
%   limit $[\mu]$ of projective classes $[\mu_n]$ in $\pl(S)$.  
% Then any
%   ending measure $\mu'$ for $\ray_\infty$ satisfies
% $$i(\mu,\mu') = 0.$$
% In particular, if $\mu$ fills the surface, then $|\mu| = |\mu'|$.

\end{proposition}

We note that Lemma~2.9 and Lemma~2.10 in \cite{\BMMI} are not stated
for segments. However the proofs are verbatim true if we allow
$\ray_n$ to be a segment. 

% \section{The low-genus cases}
% \label{low-genus}

% \begin{theorem}{low genus}
% All our conjectures hold for
% $S=S_{0,5},$  $S=S_{1,2}$, and $S=S_{2,0}$.  
% \marginpar{Really all?}
% \end{theorem}

% \sotto{Any discussion of topology of visual sphere  for this case?}

% \sotto{Some perfunctory discussion of $\xi=4$ case also? }

% \sotto{I don't know how $S_{2,0}$ goes at all}

%% file: bgimpliesbc.tex
\section{Bounded geometry implies bounded combinatorics}
\label{BG implies BC}

Let $\calK$ denote a compact subset of the moduli space $\calM(S)$.
If $\geod$ is a Weil-Petersson geodesic segment, ray or line whose
projection to $\calM$ lies in $\cal K$, then we say $\geod$ is {\em
  cobounded} or {\em $\calK$-cobounded}.  Let $\nu^\pm=\nu^\pm(\geod)$
denote the ending data of $\geod$ (markings or laminations, as in
Section \ref{preliminaries}).

\begin{theorem}{bg implies bc}
  If $\geod$ is $\calK$-cobounded then there is a $K$ depending only on
  $\calK$ so that the ending data $\nu^\pm$ of $\geod$ satisfy
  the bounded combinatorics condition:
$$
d_W(\nu^+,\nu^-) \le K
$$
where $W$ is any essential proper subsurface of $S$.
\end{theorem}

We will also deduce the following.
\begin{theorem}{bg near teich}
  For all $\epsilon>0$ there is $D>0$ so that each bi-infinite
  $\epsilon$ thick Weil-Petersson geodesic $\geod$ lies at Hausdorff
  distance $D$ in the Teichm\"uller metric from a unique Teichm\"uller
  geodesic $\tgeod$.

% Furthermore, if $\geod$ is forward-infinite then $\elam(\geod_+)$ is a filling
% uniquely ergodic lamination.
\end{theorem}

\begin{proof}[Proof of Theorem~\ref{bg implies bc}]
Let $\G_\KK$ be the set of $\KK$-cobounded Weil-Petersson geodesics
which contain 0 in their parameter interval. Note that each end
of such a geodesic is either infinitely long, with ending lamination $\lambda^+$ or
$\lambda^-$, or closed, terminating in a point of $\Teich(S)$ with
Bers marking $\nu^+$ or $\nu^-$ (that is, there are no pinching
curves). In the latter case we may consider any of the (finitely many)
Bers pants decompositions of the endpoint. 

We will now follow a compactness argument of Mosher \cite{Mosher:elc}
to establish a bound on the combinatorics asscoiated to its Bers
markings or ending laminations.

Consider the subset
$$\Gamma\subset \G_\KK\times\ML\times\ML$$ consisting of triples
$(\geod,\mu^+,\mu^-)$ such that 
\begin{itemize}
\item $\mu^+$ is a measure on the lamination
$\lambda^+$ if the forward end $\geod_+$ is infinite, and on a Bers pants decomposition
  if it is finite; and similarly for $\mu^-$ and $\geod_-$. 
\item $\mu^\pm$  have length $1$ with respect to the hyperbolic structure on 
$\geod(0)$.
\end{itemize}
We let $\Gamma$ inherit the product topology, where we put Thurston's
topology on $\ML$ and give $\G_\KK$ the topology of convergence of
parameter intervals together with uniform convergence on compact
subsets.

We first show that the action of $\Mod(S)$ on $\Gamma$ is co-compact.
Let $\KK_0$ be a compact fundamental domain for the action of
$\Mod(S)$ on the preimage of $\KK$ in $\Teich(S)$, and let
$\G_{\KK_0}$ be the set of $\geod\in\G_\KK$ with
$\geod(0)\in\KK_0$. Clearly $\G_{\KK_0}$ is compact, and every point
in $\Gamma$ can be moved by $\Mod(S)$ into $\Gamma_0 =
\Gamma\intersect\G_{\KK_0}\times\ML\times\ML$.

Let $(\geod_n,\mu^+_n,\mu^-_n)\in\Gamma_0$ be a sequence such that
$\geod_n\to\geod\in\G_{\KK_0}$. Since $\mu^\pm_n$ have length $1$ at
$\geod_n(0)$ and $\geod_n(0)\to\geod(0)$, we may conclude that, after
restricting to a subsequence, $\mu^\pm_n$ converge in $\ML$ to
$\mu^\pm$ with length 1 at $\geod(0)$.

If $\geod_+$ is finite then $\mu^+$ is a measure on a
Bers pants decomposition for the endpoint, and similarly for the
backward end $\geod_-$.

If $\geod_+$ is infinite we must show that $\mu^+$ is a measure on the
ending lamination $\lambda^+$. We claim that the length of $\mu^+_n$
is uniformly bounded on $\geod_n(t)$ for $t\ge 0$. If $\geod_n$ is
finite in the forward direction this is a consequence of convexity of
the length function. If not, then since $\geod_n$ is recurrent (being
cobounded), we can apply Lemma 4.5 of \cite{\BMMI}, ensuring that a
measured lamination has bounded length along a recurrent ray if and
only if its support is the ending lamination. Hence $\mu^+_n$ is
bounded along $\geod_n$ in the forward direction, and by convexity it
is bounded by 1 for $t\ge 0$. It follows in the limit that $\mu^+$ has
bounded length along $\geod_+$, and hence (again by Lemma 4.5 of
\cite{\BMMI}) its support is its ending lamination. The same applies
to $\mu^-$ and $\geod_-$, and so we conclude that
$(\geod,\mu^+,\mu^-)\in\Gamma_0$.  This proves that $\Gamma$ is
co-compact.

Now consider the set of quadruples $(\geod,\mu^+,\mu^-,W)$ where
$(\geod,\mu^+,\mu^-)\in \Gamma$ and $W$ is a proper essential
subsurface.  Let us first consider the case of bi-infinite geodesics:
The length $\ell_{\geod,\boundary W}(t)$ is a proper convex function
of $t$ and hence has a unique minimum. After reparameterizing and
rescaling the $\mu^\pm$ we may assume that $\boundary W$ has minimum
length at $\geod(0)$.

Suppose that our desired bound on $d_W(\mu^+,\mu^-)$ fails and there
is a sequence $(\geod_n,\mu^+_n,\mu^-_n,W_n)$, normalized in this way,
such that $d_{W_n}(\mu^+_n,\mu^-_n)\to\infty$. We will find a
contradiction. The cocompactness of $\Gamma$ tells us that, after
acting by $\Mod(S) $ and restricting to a subsequence, we can assume
that $(\geod_n,\mu^+_n,\mu^-_n)\to(\geod,\mu^+,\mu^-)$ in $\Gamma$,
which must still be bi-infinite.  We may also assume that $\{\boundary
W_n\}$ converges in $\PML(S)$, to a projectivized measured lamination
represented by $\sigma\in\ML(S)$. Continuity of length on
$\Teich(S)\times\ML(S)$ and convexity in the limit, implies that
$\ell_{\geod,\sigma}(t)$ still has a minimum at $t=0$. Hence $\sigma$
cannot have support equal to either $\mu^+$ or $\mu^-$, since Lemma
4.5 of \cite{\BMMI} ensures that a measured lamination can only be
supported on the ending lamination of a recurrent ray if its length
goes to $0$ along the ray.

Since $\mu^+$ and $\mu^-$ are filling and minimal, it follows that
they intersect $\sigma$ transversely, and that $\sigma$ cuts the
leaves of $\mu^\pm$ into segments whose lengths admit some upper
bound.  A limit of laminations in $\ML(S)$ always has support
contained in any Hausdorff limit of supports of its approximates.

Now first assume that $W_n$ are {\em not} annuli. It follows that in
the sequence $\mu^\pm_n$ are cut up by $\boundary W_n$ into pieces of
bounded length. Therefore any two of these pieces intersect a bounded
number of times (usually 0), and this bounds
$d_{W_n}(\mu^+_n,\mu^-_n)$, a contradiction.

Now assume that $W_n$ are annuli. Since $\mu^+, \mu^-$ intersect
$\sigma$ transversely, their approximates $\mu^\pm_n$ make a definite
angle with the approximates $W_n$ of $\sigma$.  There is a lower bound
on the length of the geodesic representing $W_n$.  It follows that any
lift of a leaf of $\mu_n^+$ to the annular cover corresponding to a
component of $W_n$ has intersection bounded above with any other leaf
that crosses $W_n$.  This again gives a contradiction to the
assumption that the projections go to infinity.

When $\geod$ has endpoints, the minimum of $\ell_{\geod,\boundary W}$
can occur at the endpoint, and the same can occur in the
limit. However since the geodesics are co-compact, the minimum is
bounded away from $0$.  The same argument still shows that, in the
limit, $\sigma$ cannot be supported on an ending lamination of an
infinite end. A marking intersects {\em every} lamination (the
components of the marking base are intersected by the transversals),
so the same contradiction can be obtained.
\end{proof}

\begin{proof}[Proof of Theorem \ref{bg near teich}]

%   There is a uniform distance $L$ such that for any $\geod$ longer
%   than $L$ the ending laminations bind.  We clearly can assume this
%   holds for all $\geod$.  

  By \cite{Gardiner:Masur:Extremal}, for any pair $(F_h,F_v)$ of
  measured laminations that bind $S$ there is a unique surface
  $X=X(F_h,F_v)\in\Teich(S)$ and quadratic differential
  $q=q(F_h,F_v)$, holomorphic on $X$, whose horizontal and vertical
  measured foliations are equivalent to $F_h$ and $F_v$ respectively
  (via the usual equivalence between measured foliations and
  laminations). The family $X(t) = X(e^tF_h,e^{-t} F_v)$ is a
  Teichm\"uller geodesic parameterized by arclength (and all
  Teichm\"uller geodesics are obtained this way). Note actually that
  $X(kF_h,kF_v) = X(F_h,F_v)$ for any $k>0$, since the two
  constructions differ only by a conformal factor. Hence {\em any }
  two multiples of $F_h$ and $F_v$ yield points on the same
  Teichm\"uller geodesic.

  For each $(\geod,\mu^+,\mu^-)\in \Gamma$, the laminations $\mu^+$
  and $\mu^-$ bind the surface (by Corollary~4.6 of \cite{\BMMI}) so
  we can therefore associate the (parameterized) Teichm\"uller
  geodesic
$$\tgeod(t) = X(e^t\mu^+,e^{-t}\mu^-).$$ 

%% We first find a quadratic differential $q(0)$ with vertical and
%% horizontal foliations given by $\mu^+$ and $\mu^-$.  In the finite
%% case $q(0)$ is a Strebel differential; it has closed vertical
%% trajctories in the class of the curves on $\mu^+$ and closed
%% horizontal trajectories in the class of $\mu^-$.  The heights of the
%% cylinders are given by the weights on the curves. Now $q(0)$ defines a
%% Teichm\"uller geodesic $\tgeod(t)$ with $q(0)$ the quadratic
%% differential that lies on the Riemann surface $\tgeod(0)$
 
%% ************
%% Recall also that a pair of measured foliations $F_h,F_v$ that fill the surface determines a parametrized Teichm\"uller geodesic $\tgeod$. 
%% Namely for each $t\in\reals$, the pair $e^tF_h,e^{-t}F_v$  determines a quadratic differential $q(t)$ on a Riemann surface 
%% $\tgeod(t)$.  The  
%% horizontal foliation of $q(t)$ is given by $e^tF_h$ and the vertical foliation by $e^{-t}F_v$.
%% *************

Now we wish to prove that $d_T(\geod(t),\tgeod)$ is bounded 
(uniformly on $\Gamma$).  The map that assigns to each
$(\geod,\mu^+,\mu^-)\in\Gamma$ the point $(\geod(0), \tgeod(0))\in
\Teich(S)\times \Teich(S)$ is $\Mod(S)$-equivariant and continuous on the
co-compact set $\Gamma$. Thus for some $M$ we have
$$d_T(\geod(0),\tgeod(0))\leq M$$ 
for all points in $\Gamma$.  Now let $t$ be any parameter value in the
domain of $\geod$ and define the geodesic $\geod_t$ by
$\geod_t(s)=\geod(s+t)$, so that $\geod_t(0)=\geod(t)$.  Let
$\mu_t^+,\mu_t^-$ be the multiples of $\mu^+,\mu^-$ that have length
$1$ on $\geod_t(0)$; then we have
$(\geod_t,\mu^+_t,\mu^-_t)\in\Gamma$.  The corresponding Teichm\"uller
geodesic $\tgeod_t$ satisfies $\tgeod_t(0)=\tgeod(s)$ for some $s$.
Then the above says
that we have  $$d_T(\geod(t),\tgeod(s))=d_T(\geod_t(0),\tgeod_t(0))\leq M.$$

This shows that $\geod$ lies in an $M$-neighborhood of $\tgeod$, for all
$(\geod,\mu^+,\mu^-)\in\Gamma$. It remains to obtain a bound in the
other direction. Given $\geod$ with parameter interval $J$, for each
integer point $n\in J\intersect\Z$ let $s_n$ be a point in the domain
of $\tgeod$ such that $d_T(\geod(n),\tgeod(s_n)) \le M$. Since $\geod$
is $\KK$-cobounded, the {\em Teichm\"uller} distance
$d_T(\geod(n),\geod(n+1))$ is bounded by some $M'$. Hence there is a
uniform upper bound on $d_T(\tgeod(s_n),\tgeod(s_{n+1}))$, and so
$\tgeod([s_n,s_{n+1}])$ lies in a uniform neighborhood of $\geod$,
guaranteeing that $\tgeod$ lies in a uniform 
neighborhood of $\geod$ in the Teichm\"uller metric.
% In the case with
% (one or two) endpoints the same holds for $\tgeod$ restricted to the interval
% bounded by $\inf\{s_n\}$ and $\sup\{s_n\}$. 
\end{proof}

% \begin{proof}[A second proof of Corollary \ref{bg implies bc}.]
% For the hell of it, the following is a proof 
% that does not appeal to the Teichm\"uller picture, and is just an
% extension of the argument in Theorem \ref{bg near teich}.

% \end{proof}

%% file: bcimpliesbg.tex
\section{Bounded combinatorics implies bounded geometry}
\label{BC implies BG}
In this section we will prove the converse to Corollary \ref{bg
  implies bc}, namely that a Weil-Petersson geodesic segment, ray, or
line whose end-invariants have bounded combinatorics must have bounded
geometry:

\begin{theorem}{bc implies bg}
  Given $K>0$ and a compact $\KK_0 \subset \MM(S)$, there is a
  compact $\KK\subset \MM(S)$ such that the following holds: Let
  $\geod$ be a geodesic segment ray or line with finite endpoints, if
  any,   projecting to $\KK_0$ and ending data $\nu^{\pm}$. If
\begin{equation}\label{bd combinatorics}
d_W(\nu^+,\nu^-) \le K
\end{equation}
for all proper essential subsurfaces $W$, then $\geod$ is $\KK$-cobounded.
\end{theorem}

%\marginpar{$\lambda^\pm$ should be full markings if not laminations!}

We will prove this first in the case that $\geod$ is a finite segment,
and in \S\ref{infinite case} generalize to rays and lines.  The first
step, in \S\ref{stability}, is to use a ``stability of
quasigeodesics'' argument to argue that the geodesic $\geod$ must
remain within a bounded Weil-Petersson distance of a path arising from
a hierarchy path in $P(S)$ connecting its endpoints. Indeed this will
hold not just with the general bound on projections but with the
weaker assumption of that only the non-annular projections are
bounded. Theorem \ref{i bounded} will give the combinatorial version
of this stability statement.

In \S\ref{coboundedness}-\ref{proof twist big2} we deduce the full
strength of Theorem~\ref{bc implies bg}, which in view of Theorem
\ref{i bounded} corresponds to showing that the geodesic stays away
from the strata of the completion that are combinatorially close to
the hierarchy path.  A result of Wolpert (Theorem \ref{geodesic
  limits}) will be used in \S\ref{proof twist big2} to show that under
these circumstances, close approaches to these strata force the
buildup of Dehn twists in certain curves, which (together with the
information from Theorem \ref{i bounded}) will contradict the bound on
annular projections.

%%%%%%%%  THE STABILITY OF QUASIGEODESICS ARGUMENT %%%%%%%
\input{stability}
%%%%%%%%%%%%%%%%%%%%%%%%%%%%%%%%%%%%%%%%%%%%%%%%%%%%%%%%%%

\subsection{Coboundedness}
\label{coboundedness}

With Theorem~\ref{i bounded} in hand, we return to the proof of
Theorem \ref{bc implies bg} in the finite case. Namely, we show that a
geodesic segment $\geod$ both of whose endpoints project to $\KK_0$,
with Bers markings $\nu_\pm(\geod)$ associated to its endpoints that
have bounded combinatorics, must be $\calK$-cobounded for a suitable
$\KK$.

Using just bounded combinatorics on non-annular subsurfaces, we apply
Theorem~\ref{i bounded} to show that the pants decompositions that
arise along $\geod$ uniformly fellow travel, with respect to the
Weil-Petersson metric, a hierarchy path joining Bers pants
decompositions for the endpoints of the segment.

To conclude that $\geod$ projects to a compact $\KK$, we will require
the bound on the annular projection distances $d_\gamma(\nu^+,\nu^-)$
as well.

Indeed, suppose there is a sequence of examples $\geod_n$ with
endpoints in $\KK_0$ and uniformly bounded combinatorics (condition
(\ref{bd combinatorics})), and a compact exhaustion
$\{\KK_n\subset\MM(S)\}$
for which $\geod_n$ exits $\til\KK_n$ (from now on we let $\til \KK$
and $\til \KK_n$ denote the preimages in $\Teich(S)$).
Let $\geod_n$ have endpoints $X_n^+$ and $X_n^-$.  Let $\nu_n^+ =
\nu(X_n^+)$ and $\nu_n^- = \nu(X_n^-)$ be the corresponding Bers markings
at the endpoints and let $Q_n^{\pm}=\text{base}(\nu_n^{\pm})$ the
corresponding Bers pants decompositions.  Let $\rho_n = \rho(Q_n^+,Q_n^-)$
denote hierarchy paths associated to $Q_n^{\pm}$.  By
Theorem \ref{theorem:pants:quasi} $Q \circ \geod_n$ is a quasigeodesic
of uniform quality, so we obtain from (\ref{projection short}) in
Theorem \ref{i bounded} a constant $D$ such that
\begin{equation}\label{Q projection short}
d_P(Q(\geod_n(t)),\pi_{\rho_n}(Q(\geod_n(t))) \le D.
\end{equation}

Fix $\epsilon_0$ smaller than 
$\inf_{Z \in  \calK_0}(\sys(Z))$, 
%\marginpar{make sure that $\sys$ is defined
 % correctly.}
and consider the length $L_n$ of the longest
interval $J_n$ in the domain of $\geod_n$ for which there is a curve
$\gamma_n \in \calC(S)$ with 
$\ell_{\geod_n, \gamma_n}(t) \le \epsilon_0$ for $t\in J_n$.
After passing to a subsequence, there are two cases:

\bold{Case 1:} {\em The lengths $L_n$ are unbounded}. Then there is a
family of intervals $J_n = [a_n,b_n]$ and curves $\gamma_n$ for which
every point in $\geod_n(J_n)$ is a bounded distance from the stratum
$\SSS_{\gamma_n}$. Let $x_n,$ $y_n \in \SSS_{\gamma_n}$ be the closest
points in the stratum to the endpoints $\geod_n(a_n)$ and
$\geod_n(b_n)$.  Since strata are geodesically embedded in $\Teich(S)$
\cite{Wolpert:Nielsen} we have
$d_{\SSS_{\gamma_n}}(x_n,y_n)\to\infty$.  Applying
Theorem~\ref{theorem:pants:quasi} to $\SSS_{\gamma_n}$, which is
naturally the Teichm\"uller space of the subsurface $W_n=S \setminus
\gamma_n$, we find that $$d_{P(W_n)}(Q(x_n),Q(y_n))\to\infty.$$

The distance formula (\ref{dist formula}) of Theorem~\ref{hierarchy
  paths} implies that there exist (non-annular) subsurfaces
$X_n\subseteq W_n$ such that $d_{X_n}(Q(x_n),Q(y_n))\to\infty$.

Now since $\dwp(x_n,\geod_n(a_n))$ and $\dwp(y_n,\geod_n(b_n))$ are bounded, 
and since $Q(\geod_n)$ is a bounded distance in $P(S)$ from $|\rho_n|$ (by
(\ref{Q projection short}), we
find that there are $i_n$, $j_n$ 
such that $d_{X_n}(\rho_n(i_n),\rho_n(j_n)) \to \infty$. 
But by Theorem~\ref{hierarchy paths}, part~(\ref{bded image}), 
%\marginpar{Defn is really a Theorem... -Y}
this means that
$d_{X_n}(Q_n^-,Q_n^+)$ is unbounded in $n$,
contradicting the hypothesis (\ref{bd combinatorics}).

\bold{Case 2:} {\em The lengths $L_n$ are bounded by some $L' > 0$.}
In this case, we will argue that if the systole goes to $0$ on
$\geod_n$ then Dehn twisting is building up somewhere along the
geodesic and that this buildup persists to its endpoints.  This
conclusion will contradict the bounds on annulus projections:

%%%%%%%% PREVIOUS VERSION OF THIS LEMMA:
%% \begin{lemma}{twist big2}
%%   Given positive constants $\epsilon_0$, $L$, $N$, and $\hat{t} \in
%%   (0,L)$, there is an $\epsilon$ so that the following holds.  Let
%%   $\geod \colon [0,L] \to \Teich(S)$ be a Weil-Petersson geodesic of
%%   length $L$ with the property that for each $\alpha \in \calC(S)$ we
%%   have
%% \begin{equation}
%% \max_{t \in [0,L]} \ell_{\geod,\alpha}(t) \ge \epsilon_0.
%% \label{max}
%% \end{equation}
%%   Then either 
%% \begin{enumerate} 
%% \item $\sys(\geod(\hat{t})) > \epsilon$, or 
%% \marginpar{still can't decide about this}
%% \item there is a $\gamma \in \calC(S)$ for which
%% $$d_\gamma(Q(\geod(0)),Q(\geod(L))) > N.$$
%% \end{enumerate}
%% \end{lemma}

% \begin{lemma}{twist big2}
%   Given positive constants $\epsilon_0$, $L$, $N$ and $a$, there is an
%   $\epsilon$ so that the following holds.  Let 
%   $\geod \colon [0,T] \to \Teich$ be a Weil-Petersson geodesic of
%   length $2a < T\le L$, and let $J\subset [a,T-a]$ be a subinterval
% with the property that for each $\alpha \in \calC(S)$ we
%   have
% \begin{equation}
% \max_{t \in J} \ell_{\geod,\alpha}(t) \ge \epsilon_0.
% \label{max}
% \end{equation}
%   Then either 
% \begin{enumerate} 
% \item $\inf_{t\in J}\sys(\geod(t)) \ge \epsilon$, or
% \item there is a $\gamma \in \calC(S)$ for which
% $$\inf_J\ell_{\geod,\gamma} < \ep$$
% and
% $$d_\gamma(Q(\geod(0)),Q(\geod(T))) > N.$$
% \end{enumerate}
% \end{lemma}

\begin{lemma}{twist big2}
  Given positive constants $\epsilon_0$, $L$ and $a$, let
  $\geod_n \colon [0,T_n] \to \Teich(S)$ be a sequence of Weil-Petersson geodesics of
  length $2a < T_n\le L$, and let $J_n\subset [a,T_n-a]$ be subintervals
with the property that for each $\alpha \in \calC(S)$ we
  have
\begin{equation}
\max_{t \in J_n} \ell_{\geod_n,\alpha}(t) \ge \epsilon_0.
\label{max}
\end{equation}
  Then either
\begin{enumerate} 
\item $\inf_n\inf_{t\in J_n}\sys(\geod_n(t)) > 0$, or 
\item after possibly passing to a subsequence,
there are  $\gamma_n \in \calC(S)$ for which
$$\inf_{J_n}\ell_{\geod_n,\gamma_n} \to 0 $$
and
$$d_{\gamma_n}(\nu(\geod_n(0)),\nu(\geod_n(T_n))) \to \infty.$$
\end{enumerate}
\end{lemma}

We postpone the proof of this lemma to \S\ref{proof twist big2}, and
use it now to complete the proof
of Theorem~\ref{bc implies bg} in the finite case.

Let $\geod_n(s_n)$ be a sequence of points on $\geod_n$ for which 
$\sys(\geod_n(s_n)) \to 0$, and let $J_n=[a_n,b_n]$ be minimal-length intervals containing
$s_n$ such that, for each $\gamma\in\CC(S)$,
$\max_{J_n}\ell_{\geod_n,\gamma} \ge \epsilon_0$. Such intervals exist
since $\sys > \ep_0$ at the endpoints of $\geod_n$, and 
since we are in Case~2, the length of a minimal one is at most $L'$. 

Let $K_\WP$ and $C_\WP$ denote the multiplicative and additive
quasi-isometry constants of Theorem \ref{theorem:pants:quasi}, and 
choose $K' > K_\WP((D+4)c_1' + D + C_\WP )$. 

Let $I_n=[t_n^-,t_n^+]$ be the interval containing $J_n$ satisfying
either $t^+_n - b_n = K'$ or $t^+_n$ equals the forward endpoint of
$\geod_n$ if the latter is distance less than $K'$ from $b_n$, and
similarly for $t^-_n$ and $a_n$. In particular the length of $I_n$ is
bounded by $2K'+L'$.  Note also that the distance of $J_n$ from each
endpoint of $I_n$ is uniformly bounded below:  in the case where
$t_n^+$ or $t_n^-$ is an endpoint of $\geod_n$, this follows from the
fact that those endpoints project to $\calK_0$, and that $\epsilon_0$ was
chosen strictly smaller than $\inf_{Z \in  \calK_0}(\sys(Z))$. 

We may therefore apply
Lemma~\ref{twist big2}, where $I_n$ play the role of the parameter
intervals $[0,T_n]$, to conclude (possibly passing to a subsequence)
the existence of curves $\gamma_n$ for which 
\begin{equation}
\label{gamma n twist grows}
d_{\gamma_n}(\nu(\geod_n(t_n^-)),\nu(\geod_n(t_n^+))) \to \infty
\end{equation}
and $t_n\in J_n$ such that
$$
\ell_{\geod_n,\gamma_n}(t_n)\to 0.
$$

If $\geod_n(t^+_n)$ is not an endpoint of $\geod_n$, we have
$d_\WP(\geod_n(t_n),\geod_n(t)) \ge K'$ for   $t\ge t^+_n$, so we conclude
$$
d_P(Q(\geod_n(t_n)),Q(\geod_n(t))) \ge K'/K_\WP - C_\WP.
$$
Now by  (\ref{Q projection short}), we obtain the bound
$$
d_P(Q(\geod_n(t)),\pi_{\rho_n}(Q(\geod_n(t)))) \le D. 
$$
Hence, we have
\begin{equation}\label{tn far from t}
d_P(Q(\geod_n(t_n)),\pi_{\rho_n}(Q(\geod_n(t)))) \ge K'/K_\WP - C_\WP - D.
\end{equation}
Let $v_n(t)$ denote a vertex of the main geodesic of $\rho_n$ which
lies in $\pi_{\rho_n}(Q(\geod_n(t)))$ (such a vertex exists by
definition of $\pi_{\rho_n}$).  Now since $\gamma_n$ lies in
$Q(\geod_n(t_n))=\text{base}(\mu(\geod_n(t_n))$ for large $n$
(recalling that $\ell_{\geod_n,\gamma_n}(t_n)\to 0$), by (\ref{tn far
  from t}) together with Lemma \ref{G properties} we get a lower bound
on $\CC(S)$-distance,
\begin{equation}\label{gamma far from nu}
d_S(\gamma_n,v_n(t)) \ge {\frac{1}{c_1'}}\left({\frac{K'}{K_\WP}} - C_\WP
- D\right) \ge D+4.
\end{equation}

%% \marginpar{put in more steps! consts prob wrong}

%% \begin{eqnarray*}
%% d_{S}(v_n(t), \gamma_n) &\ge& 
%% \frac{1}{c_1'} d(\pi_{\rho_n}(Q(\geod_n(t))),
%% \pi_{\rho_n}(Q(\geod_n(t_n)))) \\
%% &>& \frac{1}{c_1'}\left( \frac{K'}{K_\WP} - C_\WP -2D\right)\\
%% &>& D+4
%% \end{eqnarray*}
%% \marginpar{let's let $C_\WP$ be additive qi constant for $Q$ -- check one
%% last time...}

Now for $t \ge t^+_n$, again by (\ref{Q projection short}), we can
connect any vertex of $\mu(\geod_n(t))$ to $v_n(t)$ by a path in
$\CC(S)$ of length at most $D+2$. Hence by (\ref{gamma far from nu})
every vertex in this path has distance at least 2 from $\gamma_n$, and
therefore intersects $\gamma_n$. 

By the Lipschitz property of projections to $\calA_{\gamma_n}$, Proposition~\ref{Lipschitz
  projection}, it follows that we have
\begin{equation}\label{bound gamma n twist}
d_{\gamma_n}(v_n(t),\mu(\geod_n(t))) \le 4(D+2)
\end{equation}
for each $t \ge t_n^+$.
% \marginpar{note that for the endpoints you need to potentially choose
%   a different curve than the base geodesic curve, which could have
%   distance 1 from $\gamma_n$} 
Now, the diameter of the projection to $\AAA_{\gamma_n}$
of all the vertices of $m_n$ that are forward of $v_n(t^+_n)$ is
bounded above by $2M_2$, by Theorem~\ref{hierarchy paths},
part~\ref{bded image}.  By the triangle inequality (applying (\ref{bound gamma
n twist}) once for
$t=t_n^+$ and once for $t>t_n^+$), we have
\begin{equation}\label{bounded forward motion}
d_{\gamma_n}(\nu(\geod_n(t)),\nu(\geod_n(t^+_n))) < 8(D+2 ) + 2 M_2
\end{equation}
for all $t\ge t^+_n$.  
The same bound can be obtained for $t^-_n$ and $t\le t^-_n$.
Of course if $t^+_n$ is the forward endpoint of $\geod_n$ then
(\ref{bounded forward motion}) holds trivially, and similarly for
$t^-_n$. 

Now applying these bounds to
the endpoints  $u_n \le t^-_n$ and $w_n \ge t^+_n$ of
$\geod_n$, and using the growth inequality (\ref{gamma n twist grows}), 
we obtain
$$
d_{\gamma_n}(\nu(\geod_n(u_n)),\nu(\geod_n(w_n))) \to \infty.
$$
But this contradicts the bounded-combinatorics hypothesis on $\geod_n$.

We conclude that in fact $\geod_n$ are $\KK$-cobounded for some $\KK$. 
This concludes the proof of Theorem \ref{bc implies bg} in the case of
finite intervals, modulo Lemma \ref{twist big2}.

\subsection{Proof of Lemma \ref{twist big2}}
\label{proof twist big2}

We will apply Wolpert's discussion of limits of finite length geodesics in
the Weil-Petersson completion (\cite[Theorem 23]{Wolpert:compl}):

%%\marginpar{Have to fix so as to have variable lengths... later.}
\begin{theorem}{geodesic limits}{\bf -- Wolpert.}  {\sc (Geodesic Limits)}
Let $\geod_n \colon [0,L] \to \closure{\Teich(S)}$ be a sequence of finite
length geodesic segments of length $L$ in the Weil-Petersson
completion.  Then there exists a partition of the interval $[0,L]$ by
$0 = t_0 < t_1 < t_2 < \ldots t_{k} < t_{k+1} =L$, 
and
simplices $\sigma_0, \ldots, \sigma_{k+1}$ 
and simplices $\tau_i = \sigma_{i-1} \cap \sigma_i$ 
in $\widehat{\calC(S)}$
and a piecewise geodesic
$$\hat{\geod} \colon [0,L] \to \closure{\Teich(S)}$$
with the following properties.
\begin{enumerate}
\item $\hat{\geod} ((t_{i-1},t_{i})) \subset \calS_{\tau_i}$, $i=1,
  \ldots, k+1$,
\item $\hat{\geod}(t_i) \in \calS_{\sigma_i}$, $i = 0, \ldots, k+1$,
\item there are elements $\psi_n \in \Mod(S)$ and $\calT_{i,n} \in
  \twist(\sigma_i - \tau_i \cup \tau_{i+1})$, for $i = 1, \ldots, k$,
  so that after passing to a subsequence, $\psi_n(\geod_n([0,t_1]))$
  converges in $\closure{\Teich(S)}$ to the restriction
  $\hat{\geod}([0,t_1])$ and for each $i = 1, \ldots , k$, and $t \in
  [t_i,t_{i+1}]$, 
$$\calT_{i,n} \compos \ldots \compos \calT_{1,n} \compos \psi_n \left(
\geod_n(t) \right) \to \hat{\geod}(t)$$ as $n \to
\infty$. \marginpar{\tiny made slight change... okay? -J OK -Y}
\item The elements $\psi_n$ are either trivial or unbounded, and the
  elements $\calT_{i,n}$ are unbounded.
\end{enumerate}
The piecewise-geodesic $\hat{\geod}$ is the minimal length path in
$\closure{\Teich(S)}$ joining $\hat{\geod}(0)$ to $\hat{\geod}(L)$ and
intersecting the closures of the strata $\calS_{\sigma_1},
\calS_{\sigma_2}, \ldots, \calS_{\sigma_k}$ in order.
\end{theorem}
For convenience we define, for $i\ge 0$, 
\begin{equation}\label{define varphi}
\varphi_{i,n} = \calT_{i,n} \compos \ldots \compos \calT_{1,n}\compos
\psi_n.
\end{equation}

To understand the meaning of this somewhat technical statement, it is
helpful to focus on the case where all the $\tau_i$ are empty. In this
case the statement is that the sequence of bounded-length
geodesics is converging to a chain of segments in the interior of
$\Teich(S)$ with endpoints on various strata. Moreover, in the
approximating pictures, the geodesics approach the strata and ``wind
around'' them in the sense that the twisting parameters for at least
one curve per stratum grow without bound. This is encoded by the
twists $\calT_{i,n}$. 

See Figure~\ref{twist3} for a cartoon of this limiting process. 

\medskip

\realfig{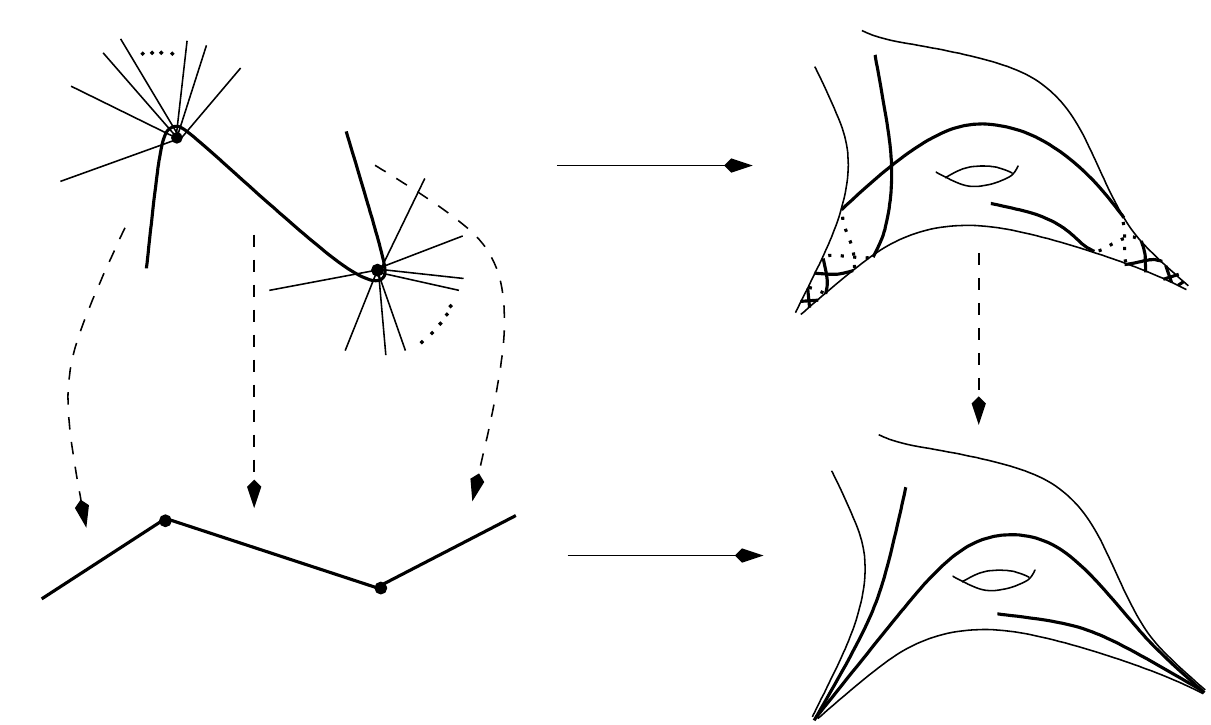}{
Geodesic limits in Teichm\"uller and Moduli space.
Horizontal arrows denote the covering from Teichm\"uller to moduli
space, and the vertical arrows denote convergence. In this
figure, $\tau_i$ are all empty.
}{3in}

\medskip

Now proceeding with the proof of Lemma \ref{twist big2},
fix positive $\epsilon_0$, $L$, and $a$. It suffices to
show, for any sequence
  $$\geod_n \colon [0,T_n] \to \Teich(S)$$ of Weil-Petersson geodesics of
  length $T_n\le L$, and intervals $J_n\subset [a,T_n-a]$ such that
\begin{enumerate}
\item $\sup_{t \in J_n} \ell_{\geod_n,\alpha}(t) > \epsilon_0$ for each
  $\alpha\in\calC(S)$ and
\item $\inf_{t\in J_n}\sys(\geod_n(t)) \to 0$,
\end{enumerate}
that, after passing to a subsequence,  there are $\gamma_n\in\calC(S)$ such that
$$
\inf_{t\in J_n}\ell_{\geod_n,\gamma_n}(t) \to 0
$$
and
$$d_{\gamma_n}(\nu(\geod_n(0)),\nu(\geod_n(T_n))) \to \infty.$$

Passing to a subsequence, trimming the intervals slightly and
changing the constants, we may assume that $T_n \equiv L$, and that 
$J_n$ converge to a subinterval $J$. Note that the lengths of $J_n$ are
bounded below since $\ell_{\geod_n,\gamma_n}$ achieves the value $\ep_0$  in $J_n$
but its infimum goes to 0; hence $J$ has positive length.

Then by Theorem~\ref{geodesic limits}, after passing again to a subsequence,
we have a partition of the interval $[0, L]$ with $0 = t_0 < \ldots <
t_k < t_{k+1} = L$, simplices $\sigma_0, \ldots, \sigma_{k+1}$, and
$\tau_1, \ldots, \tau_{k+1}$ in the curve complex $\calC(S)$, with
$\tau_i \cup \tau_{i+1} \subset \sigma_i$ and a piecewise geodesic
path
$$\hat{\geod} \colon [0,L] \to \compl,$$ for which
$\hat{\geod}([t_j,t_{j+1}])$ is a geodesic segment in the stratum
$\calS_{\tau_{j+1}}$ joining the strata $\calS_{\sigma_j}$ and
$\calS_{\sigma_{j+1}}$, and the elements $\calT_{i,n} \in
\twist(\sigma_i -  \tau_i \cup \tau_{i+1})$ are unbounded in $\Mod(S)$.
Assume the conclusions of Theorem~\ref{geodesic limits} hold, and let
$\varphi_{i,n}$ be as in (\ref{define varphi}). 

For each $i$ and $n$, let 
$$
\sigma_{i,n} = \varphi_{i,n}^{-1}(\sigma_i) = \varphi_{i-1,n}^{-1}(\sigma_i)
$$
be the pullback of $\sigma_i$ to the $\geod_n$ picture. Similarly let 
$$
\tau_{i,n} = \varphi_{i-1,n}^{-1}(\tau_i).
$$ 

We claim that $k>0$ and in fact one of the $t_i$ is contained in 
$J$. If not, then $J$ is contained within some $(t_{i-1},t_{i})$,
and so $\hat\geod(J)\subset \calS_{\tau_{i}}$. Since
$\inf_{J_n}\sys(\geod_n)\to 0$, it follows that $\tau_i$ is
nonempty. But this means that $\tau_{i,n}$ has
length going to 0 at every point of $\geod_n(J_n)$, which contradicts
property (1). 

Hence, we may fix positive $i \le k$, such that $t_i\in J$, and let $\gamma$ be a curve in
$\sigma_i \setminus (\tau_i \union\tau_{i+1})$
so that the power of the $\gamma$-Dehn twist
$\calT_\gamma$ determined by the element $\calT_{i,n}$ is unbounded
with $n$ (since the multi-twist $\calT_{i,n} \in \twist(\sigma_i -
\tau_i \cup \tau_{i+1})$ is unbounded, there exists such a $\gamma$).
Possibly passing to a subsequence, we can assume the power of
$\calT_\gamma$ tends to infinity.
Let $\gamma_n\in\sigma_{i,n}$ denote the pullback
$$
\gamma_n = \varphi_{i,n}^{-1}(\gamma).
$$

For each $i= 0, \ldots, k+1$, choose partial markings $\mu_i$ of $S$
so that
\marginpar{add ref for base curves..}
\begin{enumerate}
\item $\sigma_i \subset \base(\mu_i)$, and
\item $\mu_i$ restricts to a full marking of each component $Y \subset
  S \setminus \sigma_i$ with complexity at least one.
\end{enumerate}
Furthermore, for each $i = 0, \ldots, k$, let $\mu_i^+$ be an
enlargement of $\mu_i$ so that $\base(\mu_i^+) = \base(\mu_i)$ and
$\mu_i^+$ restricts to a full marking of each component of $S
\setminus \tau_{i+1}$ of complexity at least one. Likewise, for each
$i = 1, \ldots, k+1$, let $\mu_i^-$ be an enlargement of $\mu_i$ with
$\base(\mu_i^-) = \base(\mu_i)$ and so that $\mu_i^-$ restricts to a
full marking of each component of $S \setminus \tau_i$ of complexity
at least one. Note that $\mu_i^+$ differs from $\mu_i$ just by the
addition of transversals to the components of
$\sigma_i\setminus\tau_{i+1}$, and similarly for $\mu_i^-$ and
$\sigma_i\setminus\tau_i$. 

Further, define the pullbacks
$$
\mu_{i,n}^+ = \varphi_{i,n}^{-1}(\mu_i^+)
$$
and 
$$
\mu_{i,n}^- = \varphi_{i-1,n}^{-1}(\mu_i^-).
$$

Now we want to measure the twisting of these markings relative to
$\gamma_n$. We claim:
\begin{enumerate}
\item $d_{\gamma_n}(\mu_{i,n}^-,\mu_{i,n}^+) \to \infty$ as
  $n\to\infty$,
\item $d_{\gamma_n}(\mu_{j,n}^-,\mu_{j,n}^+)$ is bounded if $j\ne i$,
  and
\item $d_{\gamma_n}(\mu_{j,n}^+,\mu_{j+1,n}^-)$ is bounded for all
  $j$. 
\end{enumerate}
To see the first claim, note that 
$$
\varphi_{i,n}(\mu_{i,n}^-) = \calT_{i,n}(\mu_i^-).
$$ 
Thus, after applying $\varphi_{i,n}$ to all curves in our expression
we get
$$
d_\gamma(\calT_{i,n}(\mu_i^-),\mu_i^+).
$$
Now $\mu_i^-$ and $\mu_i^+$ are fixed, and each contains $\gamma$ as
well as a transversal for $\gamma$.  Since $\calT_{i,n}$ contains an
arbitrarily large power of $\calT_\gamma$, claim (1) follows.

To see claim (2), 
note that $\mu_{j,n}^+$ and $\mu_{j,n}^-$ both contain
$$
\mu_{j,n} = \varphi_{j,n}^{-1}(\mu_j) = \varphi_{j-1,n}^{-1}(\mu_j).
$$
Observe further that  $\mu_{j,n}$ contains $\sigma_{j,n}$, and in each component of
$S\setminus \sigma_{j,n}$ it restricts to a full marking. 

Now we claim that 
\begin{equation}\label{gamma in only one sigma}
\gamma_n \notin \sigma_{j,n} \ \ \text{for any $j\ne i$}.
\end{equation}
For otherwise the length of $\gamma_n$ along $\geod_n$ would
converge to $0$ both at $t_i$ and at $t_j$, and hence by convexity on
all of $[t_{i-1},t_i]$ or $[t_i,t_{i+1}]$ (the first if $j<i$ and the
second if $j>i$).  This implies that
$\gamma_n\in\tau_{i,n}$ or $\gamma_n\in\tau_{i+1,n}$, which
contradicts the choice of
$\gamma\in\sigma_i\setminus(\tau_i\union\tau_{i+1})$, so we conclude that
(\ref{gamma in only one sigma}) holds.  Thus, $\gamma_n$ intersects $\mu_{j,n}$
nontrivially, so $\pi_{\calA(\gamma_n)}(\mu_{j,n})$ is nonempty, and
it follows that the projections of the two enlargements are a bounded
distance apart in $\calA(\gamma_n)$, establishing claim (2).

To prove claim (3), note that $\mu_j^+$ and $\mu_{j+1}^-$ contain
$\tau_{j+1}$ and restrict to full markings in $S\setminus\tau_{j+1}$,
where their marking \marginpar{defined earlier?}  distance is some
finite number. Hence we may connect them with a finite sequence of
markings of the same type. Applying $\varphi_{j,n}^{-1}$, we obtain a
sequence of the same length connecting $\mu_{j,n}^+$ to
$\mu_{j+1,n}^-$, through markings that contain $\tau_{j+1,n}$ and are
full in its complement.  Since $\tau_{j+1,n}$ is contained in both
$\sigma_{j,n}$ and $\sigma_{j+1,n}$, $\gamma_n$ cannot lie in
$\tau_{j+1,n}$ by (\ref{gamma in only one sigma}). We conclude that
all the markings intersect $\gamma_n$ nontrivially, and this gives a
bound on $d_{\gamma_n}(\mu_j^+,\mu_{j+1}^-)$, as desired.
\marginpar{check end cases}

Having established all three claims, we combine them with the triangle
inequality to conclude
$$
d_{\gamma_n}(\mu_{0,n}^+,\mu_{k+1,n}^-) \to \infty.
$$
Now note that $\mu_{0,n}^+$ has bounded total length in $\geod_n(0)$
and $\mu_{k+1,n}^-$ has bounded total length in $\geod_n(L)$. It
follows that
$$
d_{\gamma_n}(\nu(\geod_n(0)),\nu(\geod_n(L))) \to \infty,
$$
as desired.

It remains to check that $$\inf_{J_n}\ell_{\geod_n,\gamma_n} \to
0.$$ 
Recall that $t_i\in J = \lim J_n$, and $\ell_{\geod_n,\gamma_n}(t_i) \to
0$. If $t_i \in J_n$ for $n$ sufficiently large, then we are done, but even if not,
note that on  $[t_i,t_{i+1}]$ the length functions
$\ell_{\geod_n,\gamma_n}$ converge uniformly to $\ell_{\hat\geod,\gamma}$,
and similarly for $[t_{i-1},t_i]$. Hence the infima on $J_n$ converge
to 0.  This concludes the proof of Lemma \ref{twist big2}.

\subsection{The infinite case}
\label{infinite case}

We are left to consider the case when
$\geod$ is bi-infinite or the case of an infinite ray $\ray$.  

Suppose a ray $\ray$ has its basepoint in $\til\calK_0$ and its ending lamination
$\lambda=\lambda^+(\ray)$ has bounded combinatorics.  In
\cite{Masur:Minsky:CCII}, it is shown that there exists an infinite
hierarchy path $\rho_\ray$ beginning at $Q(\ray(0))$ so
that $\rho_\ray(i)$ is asymptotic to $\lambda$ in $\pi(\el(S))$.
%\marginpar{thus far we have nothing about $\bdry \calC(S)$ anywhere...}
%\marginpar{Have we discussed infinite hierarchy paths?? -Y}

Letting $\mu_i = \rho_\ray(i)$ be the markings along the hierarchy path
$\rho_\ray$, we may find points $X_i$ in $\til\calK_0$ on which every
curve in $P_i$ has length bounded by some fixed $\ell$, independent of
$i$.  Letting $X_0 = \ray(0)$, the sequence of geodesic segments
$\geod_i = \geod(X_0,X_i)$ joining $X_0$ to $X_i$ projects to the
compact set $\calK$ by the above.

It follows that we may extract a limiting ray $\ray_\infty$ in the
visual sphere at $X_0$, which by 
Proposition~\ref{proposition:continuity} has ending lamination
$\lambda$ (as the lamination $\lambda$ fills the surface).  
%\marginpar{not quite: \ref{proposition:continuity} is about a sequence
% of rays, not segments...}
As each
$\geod_i$ lies in $\til\calK$, the limit $\ray_\infty$ lies in 
$\til\calK$ as well.  Then $\ray_\infty$ is recurrent, and thus by the
main theorem of \cite{\BMMI} (Theorem \ref{relc} here) we have
that $\ray_\infty = \ray$.  We 
conclude that $\ray$ lies in  $\til\calK$ as desired.

Consider a bi-infinite geodesic $\geod$, with ending laminations
$\lambda^+$ and $\lambda^-$ with $K$ bounded combinatorics.
 In
\cite{Masur:Minsky:CCII}, it is shown that there exists $\rho_\pm =
\rho(\lambda^+, \lambda^-)$, a bi-infinite hierarchy path limiting to
$\lambda^+ \in \pi(\el(S))$ in the forward direction and $\lambda^-
\in \pi(\el(S))$ in the backward direction.  We choose
$X_i^+$ and $X_i^-$ in $\til\calK_0$ on which pants decompositions $P_i^+
\in |\rho_\pm|$ and $P_i^- \in |\rho_\pm|$ have bounded length, where
$P_i^+ \to \lambda^+$ and $P_i^- \to \lambda^-$.

\realfig{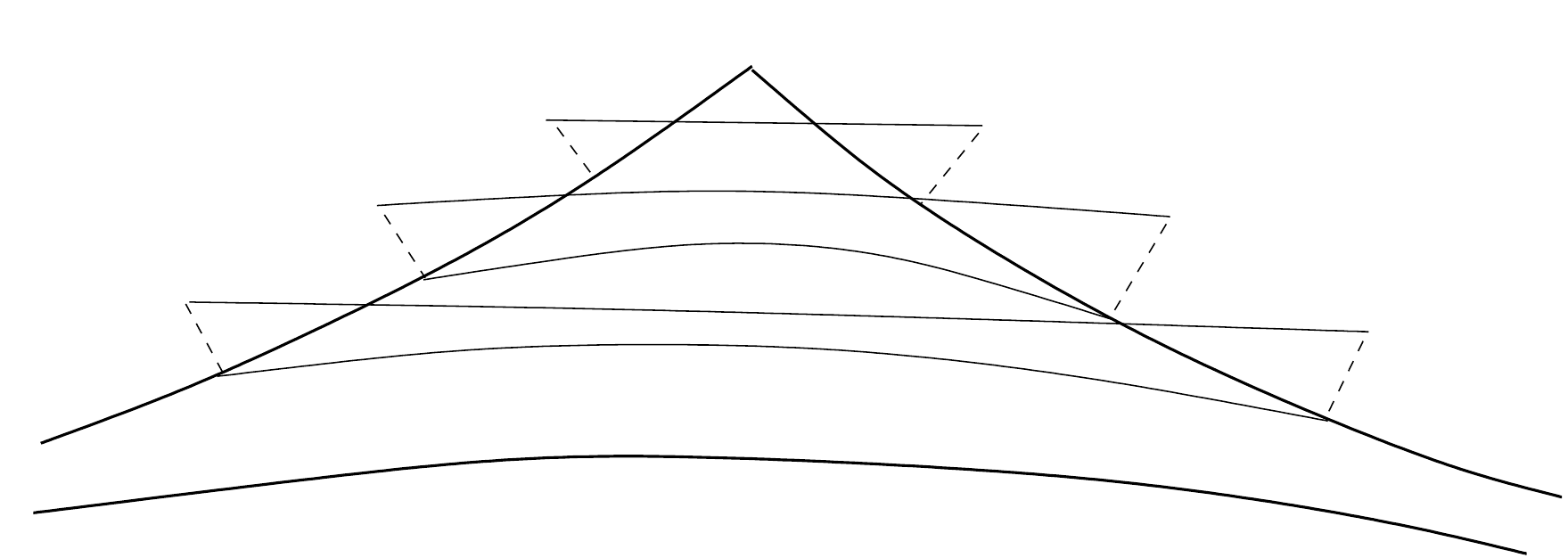}{
Extracting a bi-infinite geodesic limit.}{1.9in}

Again the geodesics $\geod(X_0,X_i^+)$ and $\geod(X_0,X_i^-)$ limit to
rays $\ray^+$ and $\ray^-$ based at $X_0$, with ending laminations
$\lambda^+$ and $\lambda^-$ by
Proposition~\ref{proposition:continuity} as above.  Each of these rays lies in
the set $\til\calK$ by the above, and is therefore recurrent.  By the
visibility property for recurrent rays, \cite[Thm. 1.3]{\BMMI}, there
is a unique bi-infinite ray $\geod_\infty$ forward asymptotic to
$\ray^+$ and backward asymptotic to $\ray^-$.  It follows that
$\geod_\infty$ has the ending laminations $\lambda^+$ and $\lambda^-$,
from which we conclude that $\geod_\infty = \geod$ by
%%\cite[Thm. 1.1]{\BMMI}.
Theorem \ref{relc}.

Now by an application of
Theorem~\ref{i bounded} to the quasi-geodesic $Q(\geod(X_0,X_i^+))$,
there is a $D>0$ so that for each fixed
$j >0$, each pants decomposition $P_j$ lies within distance $D$ of $Q(\geod(X_0,X_i^+))$
for each $i \ge j$.  It follows that $$\dwp(X_j,\geod(X_0,X_i^+)) < D' = K_\WP
D + C_\WP$$ for
each $i \ge j$.
%%%%%, where $K_\WP $ is the quasi-isometry constant of $Q$.
Thus each $X_j^+$ lies distance at most $D'$ from $\ray^+$. 
Similarly each $X_j^-$ lies distance at most $D'$ from $\ray^-$. 

Thus, if $Z_i^+$ and $Z_i^-$ are the nearest point projections of
$X_i^+$ and $X_i^-$ onto $\ray^+$ and $\ray^-$, the geodesic segments
$\geod(X_i^-,X_i^+)$ lie at a uniformly bounded distance from the
geodesic $\geod(Z_i^-,Z_i^+)$ since $\dwp$ on $\Teich(S)$ is $\CAT(0)$.  The
geodesics $\geod(Z_i^-,Z_i^+)$ converge to $\geod$ by the visibility
construction of \cite[Thm. 1.3]{\BMMI}, so it follows that
$\geod(X_i^-,X_i^+)$ converges to $\geod$ as well.  But
$\geod(X_i^-,X_i^+)$ lies in $\til\calK$ for all $i>0$, by the finite case
of Theorem~\ref{bc implies bg}, so we may conclude that $\geod$ also
lies in $\til\calK$, completing the proof.

%% file: stability.tex
\subsection{Projections to hierarchies and stability}
\label{stability}

%We will work with {\em hierarchies without annuli} (see \S 8 of
%\cite{Masur:Minsky:CCII}).
If $Q$ and $Q' \in P(S)$ are pants decompositions, we let $\rho =
\rho(Q,Q')$ denote a hierarchy path $\rho \colon [0,n] \to P(S)$, as
in Theorem \ref{hierarchy paths}, with $\rho(0)=Q$ and $\rho(n)=Q'$.
The choice of $\rho$ is not unique, but we will be satisfied with
making an arbitrary one.

% Let $\mu(H)$ denote a resolution of $H$ into a
% sequence $\{\mu_i\}$ of pants decompositions.
Let $|\rho| \subset P(S)$ denote the union 
$$
|\rho| = \cup_{i=0}^n \rho(i),
$$
namely the image of the hierarchy path in $P(S)$.  
% If $m = g_S$ is the
% main geodesic of $\rho$, 

We will at times consider pants decompositions $P$ as maximal
simplices in $\calC(S)$.  In particular, as with proper subsurfaces,
we will employ the notation
$$d_{S} (P,P') = \diam (\pi_S(P) \cup \pi_S(P'))$$
where
$$\pi_S \colon P(S) \to \calC(S)$$
denotes the projection of $P(S)$ into $\calC(S)$ that associates to a
pants decomposition $P$ the maximal simplex in $\calC(S)$ determined
by its simple closed curves, and the diameter is taken in the metric
on $\calC(S)$.
% then there is a map
% $$
% G:|\rho| \to m
% $$
% which sends each $\rho(i)$ to a vertex of $\rho(i)$ that lies in $m$ (the
% finite amount of choice in this definition will not affect our
% argument). 

The following Lemma shows that under the bounded combinatorics
assumption, the mapping $\pi_S \circ \rho$ determines a quasi-geodesic in
$\calC(S)$.
\begin{lemma}{G properties}
Given $K >0$, let $\rho$ be a hierarchy path satisfying the
non-annular $K$-bounded combinatorics condition,  namely
\begin{equation}\label{bc H}
  d_W(\rho(0),\rho(n)) \le K.
\end{equation}
for all proper non-annular essential subsurfaces $W\subset S$.
Then there is a $c_1' >1$ 
so that
\begin{equation}\label{G bounds}
  \frac{1}{c_1'}\le \frac{d_{S}(\rho(i),\rho(j))}{d_P(\rho(i),\rho(j)}\le c_1'
\end{equation}
for $i\ne j$.
\end{lemma}

\begin{proof}

  Recall the consequence~(\ref{bded hierarchy}) of
  Definition~\ref{hierarchy paths},
that we have for any
  $i$ and $j$ the bound
$$
d_W(\rho(i),\rho(j)) \le K + 2M_2 
$$
for any non-annular proper subsurface $W$.  

Taking $M_3 > K + M_1+ 2M_2 +2$, the distance formula
(Theorem~\ref{hierarchy paths}, part (\ref{dist formula})) guarantees there are $c_1$ and $c_2$
depending on $M_3$ so that we have the estimate
\begin{equation}\label{bc dist formula}
d_P(\rho(i),\rho(j)) \asymp_{c_1,c_2}
\Tsum_{V} [[d_V(\rho(i),\rho(j))]]_{M_3},
\end{equation}
where we recall the sum is over non-annular subsurfaces.
But all the terms in the sum %%%in~(\ref{bc dist formula})
%for the distance formula for $d_P(\rho(i),\rho(j))$ 
other than the $V = S$ term
are beneath the threshold $M_3$, so we have 
\begin{equation}
d_P(\rho(i), \rho(j)) 
\asymp_{c_1,c_2} d_{S} (\rho(i),\rho(j)).
\end{equation}
Since $\rho(i)$ and $\rho(j)$ are pants decompositions, the 
term $d_{S}(\rho(i),\rho(j))$ is always positive.
By Theorem~\ref{hierarchy paths}, part~(\ref{quasigeodesic}),
% we have a $K_H >1$ for which
% \begin{equation}
% \frac{1}{K_H}  \le \frac{d_P(\rho(i),\rho(j))}{|j-i|}  \le 1.
% \end{equation}
% In particular, 
$d_P(\rho(i),\rho(j))$ is always positive when $i \not= j$
so in fact there is a  $c_1'$ so that 
~(\ref{G bounds}) holds. 

%Combining constants we may take $c_0 = c_1' K_H$
%and~(\ref{G bounds}) follows.
\end{proof}

We now define a ``projection'' 
\begin{equation}\label{piH}
  \pi_\rho : P(S) \to |\rho|
\end{equation}
as follows: Given $P$ in $P(S)$ let $\beta$ be any choice of vertex
of $P$.  Let $v=\pi_m(\beta)$ be any closest point to $\beta$ in $m$, 
with respect to the metric of $\CC(S)$, where $m$ is the main geodesic
of the hierarchy path $\rho$,
and then let $\pi_\rho(P)$ be any choice of $\rho(i)$ for which
$\rho(i)$ contains $v$.
The $\delta$-hyperbolicity of $\CC(S)$ implies that $v$ is well-defined up to
uniformly bounded ambiguity, and if $\rho$ satisfies the non-annular
bound (\ref{bc H}) then
(\ref{G bounds}) implies that $\rho(i)$ is defined up to bounded
ambiguity as well.  

When $\rho$ satisfies (\ref{bc H}) we will prove that $\pi_\rho$ is
``coarsely contracting'' in the following  sense: 
\begin{theorem}{bg projection} 
Given $K$ there exist $N$, $R_0$, and $C>0$ such that if
 $\rho=\rho(Q_+,Q_-)$ is a hierarchy path satisfying
 the non-annular bounded combinatorics
 property (\ref{bc H}),  then the projection $\pi_\rho$ satisfies
\begin{enumerate}
\item For $P\in |\rho|$, $d_P(P,\pi_\rho (P))\leq N$ 
\item If $d_P(P_0,P_1)\leq 1$, then $d_P(\pi_\rho (P_0),\pi_\rho (P_1))\leq N$ 
\item If $d_P(P,|\rho|) = R \geq R_0$, then
$$
\diam\left( \pi_\rho(\NN_{R/C}(P)) \right) \le N.
$$
\end{enumerate}
Here distances and diameters are all taken in $P(S)$, and $\NN_r$
denotes a neighborhood of radius $r$ in $P(S)$. 
\end{theorem}

{\em Remark:} the result also holds in the full marking graph if we
require the bound (\ref{bc H}) for annular surfaces as well.

\begin{proof}
For Conclusion (1), we note that for $P\in |\rho|$, the diameter of
the closest point set to $P$ on $m$  has diameter at most $1$.  Then
$d_P(P,\pi_\rho(P)) \le c_1'$, by  
the lower bound in~(\ref{G bounds}). 
%in the proof of Lemma~\ref{G properties}.

Conclusion (2) follows from $\delta$-hyperbolicity of $\calC(S)$,
together with an application of ~(\ref{G bounds}).
%(\ref{G bounds}) above.
In particular, it is a standard property of $\delta$-hyperbolic spaces
that there is an $L_\delta$ depending only on the hyperbolicity
constant $\delta$ so that the nearest point projection to a
geodesic is $L_\delta$-Lipschitz. \marginpar{ref?}
If $d_P(P_0,P_1) \le 1$, then for any $\alpha \in P_0$ and $\beta \in
P_1$ we have 
\begin{equation} 
\label{Lipschitz CC}
d_S(\pi_m(\alpha),\pi_m(\beta)) < 2L_\delta.
\end{equation}
  It
follows that $d_P(\pi_\rho(P_0),\pi_\rho(P_1)) < c_1' 2 L_\delta$,
where $c_1'$ is the constant from~(\ref{G bounds}).

Let us prove (3). 

Begin with $P_0\in P(S)$ such that $d_P(P_0,\pi_\rho(P_0)) = R$, and 
consider a second pants decomposition $P_1$. 
Let $\rho' = \rho(P_0,P_1)$ be a hierarchy path from $P_0$ to $P_1$
and let $m'$ be the main geodesic of
$\rho'$. Let $P'_i=\pi_\rho(P_i)$. 
Let $v_i \in \pi_S(P_i')$
be vertices of $P'_i$ that lie on $m$ (see Figure \ref{project2}).

\realfig{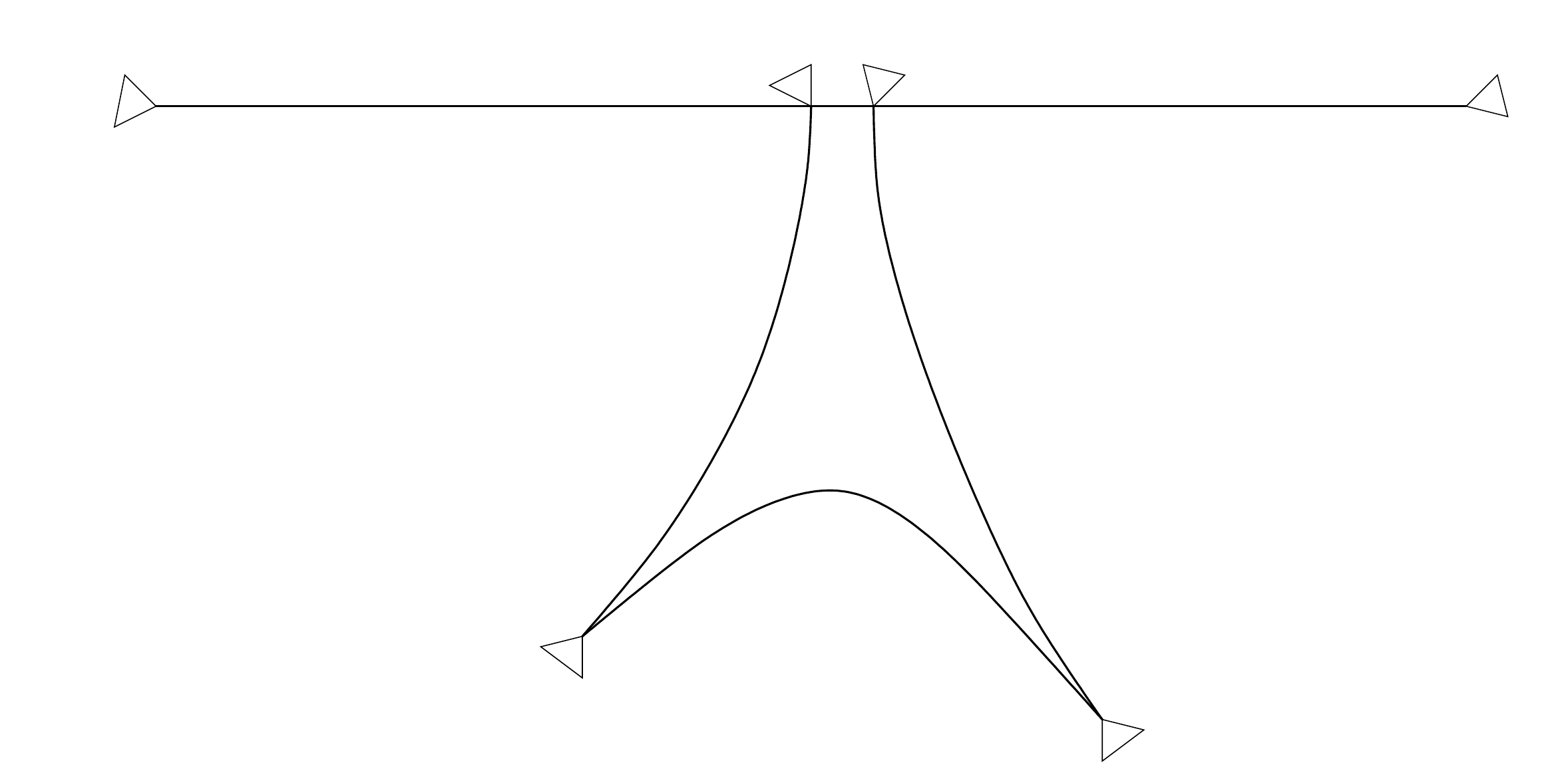}{The main geodesics of the hierarchies in
  Theorem~\ref{bg projection}.}{2in}

In view of (\ref{G bounds}), it suffices to prove that a bound of the 
form 
$d_{P}(P_0,P_1) < R/C$ implies a uniform bound on
$d_{S}(v_0,v_1)$ for $R$ at least some $R_0$.  
In particular, if there is a uniform bound 
$d_{S}(v_0,v_1) < B$ then we have
$$d_{S}(P'_0,P'_1) < B +2$$ and therefore by ~(\ref{G bounds})
%so if $P'_0 = \rho(j_0)$  and $P'_1 = \rho(j_1)$ then
%(\ref{G bounds}) guarantees that $|j_1 - j_0| < c_0 (B+2)$.
%Since $\rho$ is $1$-Lipschitz, it follows that we have the uniform bound
\begin{equation}
\label{boundpants}
d_P(\pi_\rho(P_0), \pi_\rho(P_1)) = d_P(P'_0,P'_1)  < c_1'(B+2)
\end{equation}
which is the desired conclusion for (3).  We proceed to deduce this implication.

As in Lemma~\ref{G properties}, 
we take $M_3 = M_1 +  K + 2 M_2 + 2$
% he distance formula in Definition~\ref{hierarchy paths},
%   part~\ref{dist formula} states that there are constants 
and let $c_1$, $c_2$, be constants supplied by the distance formula
of Theorem \ref{hierarchy paths} such that
%  \marginpar{More detailed ref}
  \begin{equation}\label{distance formula}
d_P(Q,Q') \asymp_{c_1,c_2}  \Tsum_W \left[[ d_W(Q,Q') \right]]_{M_3}
  \end{equation}
Let $c_3$ and $c_4$ be the constants determined by the distance
formula for the threshold constant $2M_3$.
%z
%  so that we have, in
% particular,  
% $$d(P_0,P_0') \asymp_{c_3,c_4} \Tsum_V [[d_V(P_0,P_0')]]_{2M_3}.$$
%
%   By Definition~\ref{hierarchy paths}, if $d_W(Q,Q')>M_1$ then $W$ must
%   be a component domain of $\rho$, with associated geodesic $h \in
%   \calC(W)$ whose length $|h|$ satisfies 
%   $$|h|- 2M_2 \le d_W(Q,Q') \le  |h|+2M_2.$$ 
Let $m_i$ be the main geodesic of the hierarchy path $\rho_i =
\rho(P_i,P'_i)$. 
% Applying the previous paragraph, we find that
% (assuming $R$ is sufficiently large compared to 
% %$B$ and $C_i$
% $c_2$, $M_1$, and $M_2$)
% \marginpar{Mention $R_0$ here?}  either
% \begin{equation}\label{m' large}
% |m_0| > R/4c_1
% \end{equation}
% or
% \begin{equation}
% \label{subsurfaces large}
% \Tsum_{W \subsetneq S} d_W(P_0,P'_0) > R/4c_1
% \end{equation}

By the distance formula, 
once $d_P(P_0,P_0') =R > 2c_3 c_4$, we have
$$\Tsum_V [[d_V(P_0,P_0')]]_{2M_3} \ge \frac{R}{c_3}-c_4\geq \frac{R}{2c_3}.$$
It follows that either 
\begin{equation}\label{m large}
d_S(P_0,P_0') = |m_0| > R/4c_3
\end{equation}
or
\begin{equation}
 \label{subsurfaces large}
 \Tsum_{V \subsetneq S} [[d_V(P_0,P'_0)]]_{2M_3} > R/4c_3.
\end{equation}
(The first corresponds to the $W=S$ term taking up at least half of  the
 sum in the distance formula~(\ref{distance formula}), 
while the second corresponds to the rest of the terms
 taking up at least half of the sum). 
 By hyperbolicity of $\CC(S)$, we have an $A_S$, $B_S$ and $C_S$
 depending only on the hyperbolicity constant $\delta$ for $\calC(S)$
 so that provided $|m_0| > A_S$ we have
$$
\diam \pi_m(\NN^\CC_{|m_0|/C_S}(P_0)) \le B_S
$$
where 
$\NN^\CC_r$ denotes an $r$-neighborhood with respect
to the $\CC(S)$-metric. 

Choose the constants $R_0$, $C$, and $N$ so that we have
$$R_0> 4c_3 A_S,\ \   C >
4 c_3 C_S,\ \ \text{and} \ \   N > c_1' B_S.$$

Suppose first that (\ref{m large}) holds. 
Since $$\pi_S(\NN_r(P_0)) \subset \NN^\CC_r(P_0),$$ we may conclude that
if  $R >R_0$ and 
$$
d_P(P_0,P_1)\le \frac{R}{C} < \frac{|m_0|}{C_S }
$$ 
then 
$$d_{S}(v_0,v_1) \le B_S.$$
This concludes the proof in this case. 

Now suppose that (\ref{m large})  does not hold, and thus 
(\ref{subsurfaces large}) holds.
% We claim that for $R$ large enough there must be  some nonannular
% domain $V\neq S$  such that 
% %% We need to prove that
% %% there are constants $C,B$ such that $d(P_0,P)\leq R/C$ implies
% %% $diam (\pi_m m_0)\leq B$.
% %
% %Recall now the estimate from (MM2? later?) which says that
% %\begin{equation}\label{sup dist formula}
% %      \sum_W [[d_W(Q,Q')]]_M_1 \le   C \sup_W d_W(Q,Q')^p
% %      \Tsum_W d_W(Q,Q') \le   C \sup_W d_W(Q,Q')^p
% %\end{equation}
% %where $C$ and $p$ depend only on $S$. 
% %This implies that,
% % 
% $$
%     d_V(P_0,P'_0) \ge 2M_3. 
% $$
% For otherwise the distance formula (\ref{distance formula}) would say that 
% $$|m_0|\geq c_3d(P_0,P_0'),$$ for some $c_3>0$ just depending on $M_3$.  For $R$ large enough we have contradicted that 
% (\ref{m' large}) does not hold. 

We claim that for some $B$ just depending on $S$, 
if $d_{S}(v_0,v_1)\ge B$ then for any proper subsurface $V$ of $S$ we have
\begin{equation}\label{V between}
  d_V(P_0,P_1) \ge d_V(P_0,P'_0) - M_3.
\end{equation}

%Let $Q=\pi_\rho (P)$ contain $w\in m$. Now we claim that for some $B$

To see this, we observe that if $V$ is not a component domain of
$\rho_0$ then the claim clearly holds since $M_3 > M_1$ and
$d_V(P_0,P_0') < M_1$.  
Assume, then, that $V$ is a component
domain of $\rho_0$, so that $\bdry V$ has distance at most 1 from $m_0$.

As in (\ref{Lipschitz CC}), $\delta$-hyperbolicity of $\CC(S)$ guarantees
that if there is a point $v_0'$ of $m_0$ which is within distance $2$
of a point $v_1'$ of $m_1$ then we have $d_S(v_0,v_1) \le 2 L_\delta$,
since $v_0 = \pi_m(v_0')$ and  $v_1 = \pi_m(v_1')$.
% For
% supposing otherwise, let $m'$ be the segment $\overline{v_0'v_1'}$ of length at
% most $2$. We consider the quadrilateral with sides $m_0,m_1,m'$ and
% the geodesic subsegment $\overline{v_0 v_1}$ of $m$.  Since the
% quadrilateral is $2\delta$-thin, the midpoint $v$ of $\overline{v_0
%   v_1}$ lies within $2\delta$ of $m_0\cup m_1\cup m'$ and therefore
% within $2\delta+1$ of a point $w\in m_0\cup m_1$.  Let us say $w\in
% m_0$ without loss of generality.  If $d_S(w,v_0)>2\delta+1$, then we have
% contradicted the property of $v_0$ that it is a nearest point projection to
% $m$ in $\calC(S)$ of each point on $m_0$.
% On the other hand, if $d_S(w,v_0)\leq
% 2\delta+1$, then by the triangle inequality we have $$d_S(v_0,v)\leq
% 4\delta+2,$$ a contradiction.
Thus, if we choose $B>2 L_\delta$, then $V$ cannot be a component domain of
$\rho_1 = \rho(P_1,P'_1)$ because then $\boundary V$ would be distance
at most 1 from both $m_0$ and $m_1$.  
We conclude that
$d_V(P_1,P'_1)\leq M_1$.  Since $$d_V(P'_0,P'_1)\leq K + 2 M_2$$ by the
consequence~(\ref{bded hierarchy}) of $K$-bounded combinatorics, the
triangle inequality gives the claim (\ref{V between}) by the choice of $M_3$.

% We remark that the above argument also verifies claim (2) of the
% Theorem explicitly, since assuming
% $d(P_0,P_1) \le 1$ gives a special case of the above where $v_0' \in
% \pi_S(P_0)$ and $v_1' \in \pi_S(P_1)$, and the conclusion that
% $d_S(v_0,v_1) < 8 \delta +4$ guarantees that
% $$d(\pi_\rho(P_0),\pi_\rho(P_1)) < c_1' (8 \delta + 4)$$ by an
% application of~(\ref{G bounds}).

Summing, we have
\begin{eqnarray*}
\Tsum_{V\subsetneq S}
[[d_V(P_0,P_1)]]_{M_3} &\geq&  
\Tsum_{V \subsetneq S} \left( [[d_V(P_0,P'_0)]]_{2M_3} -M_3 \right) \\
&\ge& \frac{1}{2}
\Tsum_{V \subsetneq S} [[d_V(P_0,P_0')]]_{2M_3}
\end{eqnarray*} 
But then (\ref{subsurfaces large})  gives the bound
\begin{equation}
c_1 d_P(P_0,P_1) + c_2 \ge  \frac{R}{8c_3}.
\end{equation}
Set  $c_5=16 c_3 c_1$.  Then for $R > 16 c_3 c_2$,   we have
$$
d(P_0,P_1) \ge R/c_5
.$$

Choose the constants $C$, $R_0$ and $N$ so that 
 $$C > c_5 + 4c_3 C_S,$$
$$ R_0 > 16 c_3 c_2 + 2 c_3 c_4 + 4c_3 A_S$$ and   $$N >
\max\left\{c_1'(B+2), c_1'(2L_\delta + B_S)\right\}.$$ We have shown that for
$R>R_0$, if $$d_{P}(P_0,P_1) \le R/C,$$ then
$$d_{S}(v_0,v_1)\le B.$$ This completes  the proof. 

%(8 \delta + 4) %% old version
%\}.$$
\end{proof}

Theorem \ref{bg projection} implies that a hierarchy path joining
points with uniformly bounded projection distances to all proper,
non-annular subsurfaces of $S$
produces a quasi-geodesic in $P(S)$, moreover a {\em stable} one.  
\begin{theorem}{i bounded}
For each $K$, $K_0$, there is a $D$ so that if $\rho = \rho(Q_+,Q_-)$
is a hierarchy path such that for some $K$ its endpoints satisfy
the non-annular $K$-bounded combinatorics condition and 
 if 
$$F:[0,T]\to P(S)$$
is a $K_0$-quasi-geodesic with $F(0)=Q_-$ and $F(T)=Q_+$, then we have
\begin{equation}\label{projection short}
d_{P}(F(t),\pi_\rho(F(t))) \le D.
\end{equation}
\end{theorem}
That is, any quasi-geodesic in $P(S)$ with the same endpoints as
$\rho$ must lie within a bounded neighborhood of $|\rho|$, where the
bound depends on the quality of the quasi-geodesic.  This is proven
using the usual Morse projection argument as in Mostow's rigidity theorem;
see~\cite[Lemma 6.2]{Masur:Minsky:CCI}.
%% By Theorem~\ref{theorem:pants:quasi}, $Q(\geod(t))$ is a
%% $K_\WP$-quasi-geodesic in $P(S)$. Thus we can apply
%% (\ref{projection short}) to $Q\circ \geod$, concluding that it remains
%% near $\rho$ in $P(S)$.

%% file: entropy.tex
\section{Counting closed orbits and topological entropy}
\label{entropy}
We return to consider the Weil-Petersson geodesic flow on $T^1
\Teich(S)$, the unit tangent bundle to $\Teich(S)$ and the flow on the
quotient $\MM^1(S)$.  We recall from \cite{\BMMI} that the geodesic
flow is not everywhere defined, but is defined on the full
Liouville-measure subset $\calF \subset \calM^1(S)$ consisting of
lifts of bi-infinite geodesics to the unit tangent bundle.

Given any compact flow-invariant subset $\KK$ of $ \MM^1(S)$,
the question of the topological entropy $h_{\rm top}(\KK)$ of the flow
$\varphi^t$ can be formulated.
In this section we show
\begin{named}{Theorem~\ref{theorem:entropy}}
\entropy
\end{named}

The estimate of entropy follows directly from estimates on the
asymptotic growth rate of the number of closed orbits of the geodesic
flow in a compact set.  Given a compact subset $\calK \subset
\calM^1(S)$, let $n_\calK(L)$ denote the number of closed orbits of
the geodesic flow of length at most $L$ that are contained in
$\calK$. We are interested in the asymptotic growth rate
$$
p_\varphi(\calK) 
= \liminf_{L \to \infty} \frac{\log n_\calK(L) }{L}.
$$ 

\begin{named}{Theorem~\ref{growth rate}}{\sc (Counting Orbits)}
\growthrate
\end{named}

The relationship between the conclusions of Theorem~\ref{growth rate} 
and Theorem~\ref{theorem:entropy} for $\varphi^t$ lies in
Proposition~\ref{growth and entropy}, below, once we have shown that $\varphi^t$ restricted
to any compact invariant subset is {\em expansive}.
\begin{definition}{expansive}
A flow $\varphi^t$ on a metric space $(X,d)$ is {\em expansive} if
there is a constant $\delta > 0$ so that the following property holds.
Suppose $f\colon \reals\to\reals$ is any continuous surjective function 
with $f(0) = 0$ and such that $d(\varphi^t(x) , \varphi^{f(t)}(x)) <
\delta$ for all $x,t$.  Then if $x,y$ are such that 
$$
d(\varphi^t(x),\varphi^{f(t)}(y)) < \delta
$$
for all $t$, then there is a $t_0$ so that $\varphi^{t_0}(x) = y$.
\end{definition}
We defer the proof that $\varphi^t$ is expansive to Lemma~\ref{flow
  expansive} and proceed to the proof of Theorem~\ref{growth rate},
from which we will derive Theorem~\ref{theorem:entropy} as a
consequence.  To do so, we note the following.

\begin{proposition}{growth and entropy}{\rm
    \cite{Hasselblatt:Katok:book}} 
Let $\varphi^t$ be an expansive flow on a  metric space
$(X,d)$, and let $\calK$ be a compact invariant subset for $\varphi$.
Then we have
$$
p_\varphi(\calK) < h_{\rm top}(\calK).
$$
\end{proposition}

\begin{proof}[Proof of Theorem~\ref{growth rate}]
  We define a family of compact invariant subsets for the
  Weil-Petersson geodesic flow with a larger and larger exponential
  growth rate for the number of closed orbits of length at most $L$.

  Given $K > 0$, we let $\calF_K \subset T^1 \Teich(S)$ denote the
  collection of lifts to $T^1 \Teich(S)$ of bi-infinite geodesics in
  $\Teich(S)$ with $K$-bounded combinatorics; in other words if $\geod
  \in \calF_K$ then we have
$$d_Y(\elam^+(\geod),\elam^-(\geod)) \le K$$ for each proper essential subsurface $Y
\subsetneq S$ that is not a three-holed sphere.  Then by
Theorem~\ref{bc implies bg}, $\calF_K$ projects into a compact subset
of $\calM(S)$ and hence a compact subset of $\calM^1(S)$, so the
closure $\overline \calF_K$ has compact image in its projection to
$ \calM^1(S)$.

For $K$ sufficiently large, the set $\calF_K$ contains pseudo-Anosov
axes by \cite{Daskalopoulos:Wentworth:mcg}, 
and thus its projection to $\calM^1(S)$ contains closed
orbits.  Let $\calO_K$ denote the collection of closed orbits in the
projection of $\calF_K$ to $\calM^1(S)$.  Then the closure $\closure
\calO_K$ is a compact geodesic-flow-invariant  subset of $\calM^1(S)$.

The asymptotic growth rate for the number of closed geodesics in
$\calO_K$ can estimated from below by a direct construction of a
family of pseudo-Anosov elements of $\Mod(S)$ with $K$-bounded
combinatorics.  

We build this family using a construction of Thurston
\cite[Thm. 7]{Thurston:mcg} as follows.  A pair of (isotopy classes
of) simple closed curves $\alpha$ and $\beta$ {\em bind the surface}
$S$ if given representatives $\alpha^*$ and $\beta^*$ on $S$ for which
$i(\alpha, \beta) = |\alpha^* \cap \beta^*|$ each component of $S
\setminus ( \alpha^* \cup \beta^*$) is either a disk or an annulus
that retracts to a boundary component of $S$. 

The pair of curves determines a {\em Teichm\"uller disk},
$\Delta_{(\alpha, \beta)}$ an isometrically embedded copy of $\half^2$
(in the Teichm\"uller metric), and a representation $\rho$ of the
group $\langle \tau_\alpha, \tau_\beta \rangle$ generated by Dehn
twists about $\alpha$ and $\beta$ into the stabilizer of
$\Delta_{(\alpha, \beta)}$ in $\Mod(S)$, which naturally acts
isometrically on $\Delta_{(\alpha, \beta)}$.  The representation
$\rho$, which simply restricts the Dehn-twists as isometries of
$\Teich(S)$ to the disk $\Delta_{(\alpha,\beta)}$, has the property
that a given $\varphi \in \Mod(S)$ is of finite order, reducible, or
pseudo-Anosov, according to whether it has image $\rho(\varphi)$
an elliptic, parabolic, or hyperbolic element of $\PSL_2(\reals)$.

In Thurston's construction, $\rho$ sends the Dehn twists $\tau_\alpha$
and $\tau_\beta$ to the elements
\begin{displaymath}
\rho(\tau_\alpha) = \left[ 
\begin{array}{cc}
1 & k \\
0 & 1 
\end{array}\right]
\ \ \ \text{and} 
\ \ \ 
\rho(\tau_\beta) = \left[ 
\begin{array}{cc}
1 & 0 \\
-k & 1 
\end{array}\right]
\end{displaymath}
where $k = i(\alpha,\beta)$ is the intersection number for the binding
pair $(\alpha,\beta)$.

As the trace $\tr(\rho(\tau_\alpha \compos \tau_\beta^{-1}))$ is
greater than $2$, the element $\tau_\alpha \compos \tau_\beta^{-1}$ is
pseudo-Anosov.  Thurston observes, moreover, that for $k\ge2$ the group
$\rho(\langle \tau_\alpha, \tau_\beta \rangle)$ is free and $\rho$ is
faithful, so $\langle \tau_\alpha, \tau_\beta \rangle$ is free.

Given $n$ positive integers $q_1, \ldots, q_{n}$, and letting
\begin{equation}
\psi_{(q_1,\ldots,q_{n})} 
= \tau_\alpha^{q_1} \compos \tau_\beta^{-q_1} \compos \ldots \compos
\tau_\alpha^{q_n} \compos \tau_\beta^{-q_n}
\label{pseudos}
\end{equation}
one may compute directly that $\tr( \rho(\psi_{(q_1,\ldots,q_{n})}) )$
is strictly greater than $2$ and therefore that $\psi =
\psi_{(q_1,\ldots,q_{n})} $ is pseudo-Anosov.

Since $\langle \tau_\alpha, \tau_\beta \rangle$ is free, given $q_j
\in [1,B]$ for $B > 1$, the conjugacy class of the element
$\psi_{(q_1,\ldots,q_{n})}$ is uniquely determined by the $n$-tuple
$\{q_1,\ldots,q_{n}\}$ up to cyclic permutation, so the the number of
distinct conjugacy classes of pseudo-Anosov mapping classes that arise
from this construction is $B^n/n$.

We claim the combinatorics of the stable and unstable laminations for
$\psi$ are bounded in terms of $B$ so there is a $K = K(B)$ for which
the axes of all such pseudo-Anosov mapping classes lie in $\calF_K$.
To see this, note that the attracting and repelling fixed points for
the hyperbolic element $\rho(\psi)$ in $\PSL_2(\reals)$ are real
numbers with continued fraction expansion whose entries are bounded by
$kB$ (see \cite{Series:farey}), and thus the axis projects into a
compact subset of $\half^2 /\rho(\langle \tau_\alpha, \tau_\beta
\rangle )$ depending only on $B$.  Since the inclusion
$\Delta_{(\alpha,\beta)} \includesin \Teich(S)$ is an isometry for the
Teichm\"uller metric, the axis for $\rho(\psi)$ includes as a geodesic
into $\Teich(S)$ representing the invariant axis for $\psi$ in the
Teichm\"uller metric.  As the axis projects to a compact subset of
$\calM(S)$, it follows from the main theorem of
\cite{Rafi:Teich:model} that its stable and unstable laminations have
bounded combinatorics, with bound $K = K(B)$ depending only on $B$.

By the upper bound on Weil-Petersson distance in
Theorem~\ref{theorem:pants:quasi}, there is a constant $C$ so that the
Weil-Petersson translation distance of $\psi$ is bounded above by
$nC$.  It follows that the asymptotic growth rate
$p_\varphi(\closure{\calO_K})$ of the number of closed geodesics in
$\closure{\calO_K}$ is bounded below by
$$\liminf_{n \to \infty}\frac{n \log(B) - \log(n)}{nC}$$  
which tends to infinity with $B$.  Thus the family of compact sets
$\closure{\calO_{K(B)}}$, has arbitrarily large asymptotic growth
rates for their periodic orbits.
\end{proof}

% We note a corollary of these estimates.
% \begin{corollary}{not discrete}
% The length spectrum for the Weil-Petersson metric on the moduli space
% $\calM(S)$ is not discrete.
% \end{corollary}

% \begin{proof}
% Fixing $n$, and letting the coefficients $q_j$,  $j= 1, \ldots, n$, grow without bound,
% we obtain an infinite family of pseudo-Anosov elements with a uniform
% bound on their translation distance.
% \end{proof}

% \bold{Remark:} The  authors thank Benson Farb for pointing out
% the Corollary as a consequence of Theorem~\ref{theorem:pants:quasi}.

We now show the following.
\begin{lemma}{flow expansive}
The restriction of the Weil-Petersson geodesic flow to any compact
invariant set is expansive.
\end{lemma}

\begin{proof}

%s an $\epsilon_K$ so that the injectivity radius of  $\calM(S)$ is bounded
%below by $\epsilon_k$ for each $x \in K$.   It follows that if $x$ and
%$y$ are two points in $K$ at distance at most $\epsilon_K/2$, there is a
%unique geodesic joining $x$ and $y$ of length less than
%$\epsilon_K/2$.

  Let $K$ be a compact invariant subset of $\calM^1(S)$.  There is
  $\delta=\delta(K)>0$ such that if $x$ and $y$ lie in $\calM^1(S)$ are a pair
  of points with $d(x,y)<\delta$ and $\tilde x$ is a lift of $x$ to
  $T^1 \Teich(S)$, then there exists a unique lift $\tilde y$ of $y$
  to $T^1 \Teich(S)$, such that $d(\tilde{x},\tilde{y})<\delta$.

Assume now we have a continuous
surjective function $f \colon \reals \to \reals$ for which
$$
d(\varphi^t(x),\varphi^{f(t)}(x)) < \delta.
$$
Consider any $x,y$ with the property that 
$$
d(\varphi^t(x),\varphi^{f(t)}(y)) < \delta
$$ for all $t$.

Let $\tilde \varphi^t(x)$ be a lift of $\varphi^t(x)$, and for each
$t$ find the unique vector $w(f(t))$ in $T^1 \Teich(S)$ such that
$$
d(\tilde\varphi^t(x),w(f(t)))<\delta
$$ and such that
$\varphi^{f(t)}(y)$ is the projection of $w(f(t))$ to $\calM^1(S)$.
This gives a path $w(f(t))$ in $T^1 \Teich(S)$.  Since its projection
to $ \calM^1(S)$ yields a geodesic in $\calM(S)$, it follows that
$w(f(t))$  projects to a geodesic in $\Teich(S)$.  

As $w(f(t))$ and
$\tilde\varphi^t(x)$ remain
at uniformly bounded distance in $T^1 \Teich(S)$, we conclude that
their projections to $\Teich(S)$ remain at bounded distance as well.

Since distinct bi-infinite Weil-Petersson geodesics in $\Teich(S)$
diverge, in either forward or backward time, we conclude that
$\tilde\varphi^t(x)$ and $w(f(t))$ are parametrizations by arclength of the same
geodesic, and thus we may conclude that there is a $t_0$ for which 
$$
\varphi^{t_0} (x) = y.
$$
It follows that the restriction of the flow to $K$ is expansive.

%and let $\gamma_t\colon [0,1] \to T^1 \calM(S)$ be %
%the unique geodesic segment in $\calM(S)$ joining
%$\varphi^t(x)$ to 
%$\varphi^{f(t)}(y)$ of length less than $\delta$, parametrized
%proportionally to arclength.

%Then for any interval $[a,b]$, the segments $\gamma_s$, $s \in [a,b]$, 
%lift to $T^1 \Teich(S)$ to give a geodesic homotopy from the lift of
%the geodesic segment $\{ \varphi^s(x) \st s \in [a,b]\}$ to $T^1
%\Teich(S)$ to the lift of the geodesic segment
%$\{\varphi^{f(s)}(y) \st s \in [a,b] \}$.  
%In particular, these
%geodesics remain a bounded distance $\delta$ apart in the unit tangent
%bundle to Teichm\"uller space.

%Let $\geod_x(t) = \wt \varphi^t(x)$ and $\geod_y(f(t)) = \wt
%\varphi^{f(t)}(y)$ denote the lifts of the trajectories of $x$ and $y$
%under the flow to the unit tangent bundle $T^1 \Teich(S)$ to $\Teich(S)$.
%The geodesic homotopy between the trajectories of $x$ and $y$ in $T^1
%\calM(S)$ lifts to a geodesic homotopy between restrictions of
%$\geod_x$ and $\geod_y$ to corresponding parameter intervals.  It
%follows that 
%$$
%d(\geod_x(t),\geod_y(f(t)))< \delta
%$$

%Since the distance function between geodesics is convex, it follows
%that there is a constant $\delta'$ with $0 \le \delta' < \delta$ for which
%$$
%d(\geod_x(t) ,\geod_y(t)) = \delta'
%$$
%for all $t$.  
%\marginpar{actually, we don't really use compactness in the proof,
 % just in the application I think}
\end{proof}

\begin{proof}[Proof of Theorem~\ref{theorem:entropy}]
Theorem~\ref{theorem:entropy} follows immediately as a
direct consequence of Theorem~\ref{growth rate},
Lemma~\ref{flow expansive} 
and Proposition~\ref{growth and entropy}.
%   By Theorem~\ref{growth rate}, we may, for each $N>0$ find a compact
%   invariant subset $\calK$, with the property that $p_\varphi(\calK)
%   >N$.  It follows from Lemma~\ref{flow expansive} and
%   Proposition~\ref{growth and entropy} that
% $$h_{\rm top}(\calK) > N.$$
% Since $N$ is arbitrary, this suffices to prove the theorem.
\end{proof}

%% file: wpmain.bbl
\begin{thebibliography}{Ham3}

\bibitem[Ab]{Abikoff:degenerating}
W.~Abikoff.
\newblock {Degenerating families of Riemann surfaces}.
\newblock {\em Annals of Math.} {\bf 105}(1977), 29--44.

\bibitem[Brs]{Bers:nodes}
L.~Bers.
\newblock {Spaces Of Degenerating Riemann Surfaces}.
\newblock In {\em Discontinuous Groups And Riemann Surfaces}, pages 43--55.
  Annals Of Math Studies 76, Princeton University Press, 1974.

\bibitem[Br]{Brock:wp}
J.~Brock.
\newblock {The Weil-Petersson metric and volumes of 3-dimensional hyperbolic
  convex cores}.
\newblock {\em J. Amer. Math. Soc.} {\bf 16}(2003), 495--535.

\bibitem[BF]{Brock:Farb:rank}
J.~Brock and B.~Farb.
\newblock {Rank and curvature of Teichm\"uller space}.
\newblock {\em Amer. J. Math.} {\bf 128}(2006), 1--22.

\bibitem[BM]{Brock:Masur:pants}
J.~Brock and H.~Masur.
\newblock {Coarse and synthetic Weil-Petersson geometry: quasi-flats,
  geodesics, and relative hyperbolicity}.
\newblock {\em Geometry and Topology} {\bf 12}(2008), 2453--2495.

\bibitem[BMM]{Brock:Masur:Minsky:asymptoticsI}
J.~Brock, H.~Masur, and Y.~Minsky.
\newblock {Asymptotics of Weil-Petersson geodesics I: ending laminations,
  recurrence, and flows}.
\newblock {\em Geom. Funct. Anal.} {\bf 19}(2010), 1229--1257.

\bibitem[Bus]{Buser:book:spectra}
P.~Buser.
\newblock {\em Geometry and Spectra of Compact Riemann Surfaces}.
\newblock Birkhauser Boston, 1992.

\bibitem[Chu]{Chu:noncompleteness}
Tienchen Chu.
\newblock {The {W}eil-{P}etersson metric in the moduli space}.
\newblock {\em Chinese J. Math.} {\bf 4}(1976), 29--51.

\bibitem[DW]{Daskalopoulos:Wentworth:mcg}
G.~Daskolopoulos and R.~Wentworth.
\newblock {Classification of Weil-Petersson isometries}.
\newblock {\em Amer. J. Math.} {\bf 125}(2003), 941--975.

\bibitem[GM]{Gardiner:Masur:Extremal}
F.~Gardiner and H.~Masur.
\newblock {Extremal length geometry of Teichm\"uller space}.
\newblock {\em Complex Variables, Theory and Applications} {\bf 16}(1991),
  209--237.

\bibitem[Ham1]{Hamenstaedt:boundary}
U.~Hamenst{\"a}dt.
\newblock {Train tracks and the {G}romov boundary of the complex of curves}.
\newblock In {\em Spaces of Kleinian groups}, volume 329 of {\em London Math.
  Soc. Lecture Note Ser.}, pages 187--207. Cambridge Univ. Press, Cambridge,
  2006.

\bibitem[Ham2]{Hamenstaedt:compact}
U.~Hamenst{\"a}dt.
\newblock {Dynamics of the Teichm{\"u}ller flow on compact invariant sets}.
\newblock {\em Preprint, {\tt arXiv:0705.3812}} (2007).

\bibitem[Ham3]{Hamenstaedt:compact:wp}
U.~Hamenst{\"a}dt.
\newblock {Invariant measures for the Weil-Petersson flow}.
\newblock {\em Preprint} (2008).

\bibitem[KH]{Hasselblatt:Katok:book}
A.~Katok and B.~Hasselblatt.
\newblock {\em Introduction to the modern theory of dynamical systems},
  volume~54 of {\em Encyclopedia of Mathematics and its Applications}.
\newblock Cambridge University Press, Cambridge, 1995.
\newblock With a supplementary chapter by Katok and Leonardo Mendoza.

\bibitem[Kla]{Klarreich:boundary}
E.~Klarreich.
\newblock {{The boundary at infinity of the curve complex and the relative
  Teichm\"uller space}}.
\newblock {\em Preprint} (1999).

\bibitem[MM1]{Masur:Minsky:CCI}
H.~Masur and Y.~Minsky.
\newblock {Geometry of the complex of curves I: hyperbolicity}.
\newblock {\em Invent. Math.} {\bf 138}(1999), 103--149.

\bibitem[MM2]{Masur:Minsky:CCII}
H.~Masur and Y.~Minsky.
\newblock {Geometry of the complex of curves II: hierarchical structure}.
\newblock {\em Geom. \& Funct. Anal.} {\bf 10}(2000), 902--974.

\bibitem[Mj]{Mitra:ibounded}
M.~Mj.
\newblock {Cannon-Thurston maps, i-bounded geometry and a theorem of McMullen}.
\newblock {\em Preprint, {\tt arXiv:math/0511104}} (2006).

\bibitem[Msh]{Mosher:elc}
L.~Mosher.
\newblock {Stable Teichm\"uller quasigeodesics and ending laminations}.
\newblock {\em Geom. Topol.} {\bf 7}(2003), 33--90.

\bibitem[Raf]{Rafi:Teich:model}
K.~Rafi.
\newblock {A combinatorial model for the Teichm\"uller metric}.
\newblock {\em To appear, Geom. \& Funct. Anal.}

\bibitem[Ser]{Series:farey}
C.~Series.
\newblock {Geometrical methods of symbolic coding}.
\newblock In {\em Ergodic Theory, Symbolic Dynamics, and Hyperbolic Spaces},
  pages 125--152. Oxford University Press, 1991.

\bibitem[Th]{Thurston:mcg}
W.~P. Thurston.
\newblock {On the geometry and dynamics of diffeomorphisms of surfaces}.
\newblock {\em Bull. AMS} {\bf 19}(1988), 417--432.

\bibitem[Wol1]{Wolpert:noncompleteness}
S.~Wolpert.
\newblock {{Noncompleteness of the Weil-Petersson metric for Teichm\"uller
  space}}.
\newblock {\em Pacific J. Math.} {\bf 61}(1975), 573--577.

\bibitem[Wol2]{Wolpert:Nielsen}
S.~Wolpert.
\newblock {Geodesic length functions and the Nielsen problem}.
\newblock {\em J. Diff. Geom.} {\bf 25}(1987), 275--296.

\bibitem[Wol3]{Wolpert:compl}
S.~Wolpert.
\newblock {Geometry of the {W}eil-{P}etersson completion of {T}eichm\"uller
  space}.
\newblock In {\em Surveys in differential geometry, Vol.\ VIII (Boston, MA,
  2002)}, Surv. Differ. Geom., VIII, pages 357--393. Int. Press, Somerville,
  MA, 2003.

\end{thebibliography}
